[1994/12/01]
\documentclass[10pt, reqno]{amsart}
\usepackage{a4wide}
\usepackage[english, activeacute]{babel}
\usepackage{amsmath,amsthm,amsxtra}
\usepackage{epsfig}
\usepackage{amssymb}
\usepackage{latexsym}
\usepackage{amsfonts}
\usepackage{hyperref}
\usepackage{pdftricks}

\title{Nonlinear stability of mKdV breathers}
\author{Miguel A. Alejo}
\author{Claudio Mu\~noz}
\address{Department of Mathematical Sciences, University of Copenhagen,  Denmark}
\email{miguel.alejo@math.ku.dk}
\address{Department of Mathematics, The University of Chicago, Chicago IL, USA}
\email{cmunoz@math.uchicago.edu}
\date{February, 2012}
\thanks{}

\chardef\bslash=`\\ 

\newtheorem{thm}{Theorem}[section]
\newtheorem{cor}[thm]{Corollary}
\newtheorem{lem}[thm]{Lemma}
\newtheorem{prop}[thm]{Proposition}

\newtheorem{defn}[thm]{Definition}

\theoremstyle{remark}
\newtheorem{rem}{Remark}[section]

\numberwithin{equation}{section}

\newcommand{\R}{\mathbb{R}}

\newcommand{\Z}{\mathbb{Z}}

\newcommand{\la}{\lambda}

\newcommand{\al}{\alpha}
\newcommand{\bt}{\beta}
\newcommand{\ga}{\gamma}

\newcommand{\spawn}{\operatorname{span}}
\newcommand{\sech}{\operatorname{sech}}

\def\bm{\left( \begin{array}{cc}}
\def\endm{\end{array}\right)}

 \providecommand{\abs}[1]{\lvert#1 \rvert}

\newcommand{\be}{\begin{equation}}
\newcommand{\ee}{\end{equation}}
\newcommand{\ba}{\begin{equation*}}
\newcommand{\ea}{\begin{equation*}}
\newcommand{\bea}{\begin{eqnarray}}
\newcommand{\eea}{\end{eqnarray}}
\newcommand{\bee}{\begin{eqnarray*}}
\newcommand{\eee}{\end{eqnarray*}}
\newcommand{\ben}{\begin{enumerate}}
\newcommand{\een}{\end{enumerate}}
\newcommand{\nonu}{\nonumber}

\newcommand{\eval}[2][\right]{\relax
  \ifx#1\right\relax \left.\fi#2#1\rvert}

\let\abs=\envert

\begin{document}
\begin{abstract}
Breather solutions of the modified Korteweg-de Vries equation are shown to be globally stable in a \emph{natural} $H^2$ topology. Our proof introduces a new Lyapunov functional, at the $H^2$ level, which allows to describe the dynamics of small perturbations, including oscillations induced by the periodicity of the solution, as well as a direct control of the corresponding instability modes. In particular, degenerate directions are controlled using low-regularity conservation laws. 
\end{abstract}
\maketitle \markboth{Stability of breathers} {Miguel A. Alejo and Claudio Mu\~noz}
\renewcommand{\sectionmark}[1]{}

\section{Introduction}

\medskip

This paper deals with the nonlinear stability of \emph{breathers} of the focusing, modified Korteweg-de Vries (mKdV) equation
\be\label{mKdV}
u_{t}  +  (u_{xx} + u^3)_x =0.
\ee
Here $u=u(t,x)$ is a real-valued function, and $(t,x)\in \R^2$.  The equation above is a well known \emph{completely integrable} model \cite{Ga,AC,La}, with infinitely many conserved quantities, and a suitable Lax-pair formulation. The Inverse Scattering Theory has been applied by many authors in order to describe the behavior of solutions in generality, see e.g. \cite{AC,La} and references therein.    

\medskip

Solutions $u(t,x)$ of (\ref{mKdV}) are invariant under space and time translations, and under suitable scaling properties. Indeed, for any $t_0, x_0\in \R$, and $c>0$, both $u(t-t_0, x-x_0)$ and $c^{1/2} u(c^{3/2} t, c^{1/2} x)$ are solutions of (\ref{mKdV}).  Finally, if $u(t,x)$ is a solution of (\ref{mKdV}), then $u(-t,-x)$ and $-u(t,x)$ are also solutions. 

\medskip

On the other hand, standard conservation laws for (\ref{mKdV}) at the $H^1$-level are the \emph{mass}
\be\label{M1}
M[u](t)  :=  \frac 12 \int_\R u^2(t,x)dx = M[u](0),
\ee 
and \emph{energy} 
\be\label{E1}
E[u](t)  :=  \frac 12 \int_\R u_x^2(t,x)dx -\frac 14 \int_\R u^4(t,x)dx = E[u](0).
\ee 
A satisfactory Cauchy theory is also present at such a level of regularity or even lower, see e.g. Kenig-Ponce-Vega \cite{KPV}, and Colliander {\it et al.} \cite{CKSTT}. From the Inverse Scattering Theory, the evolution of a rapidly decaying initial data can be described by purely algebraic methods. Solutions are shown to decompose into a very particular set of solutions (see Schuur \cite{Sch}),  described in detail below.

\medskip

Indeed, equation (\ref{mKdV}) is also important because of the existence of solitary wave solutions called \emph{solitons}. These profiles are often regarded as minimizers of a constrained functional in the $H^1$-topology. For example, mKdV (\ref{mKdV}) has solitons of the form
\be\label{Sol}
u(t,x) = Q_c (x-ct), \quad Q_c(s) := \sqrt{c} Q(\sqrt{c} s), \quad c>0,
\ee
with
\[
Q (s):= \frac{\sqrt{2}}{\cosh (s)} =2\sqrt{2} \partial_s[\arctan(e^{s})].
\]
By replacing (\ref{Sol}) in (\ref{mKdV}), one has that $Q_c>0$ satisfies the nonlinear ODE
\be\label{ecQc}
Q_c'' -c\, Q_c +Q_c^3=0, \quad Q_c\in H^1(\R).
\ee
Moreover, as a consequence of the integrability property, these nonlinear modes interact elastically during the dynamics, and no dispersive effects are present at infinity. In particular, even more complex solutions are present, such as \emph{multi-solitons} (explicit solutions describing the interaction of several solitons \cite{HIROTA1}).  For example, the 2-soliton solution of (\ref{mKdV}) is given by the four-parameter family 
\[
U_2  :=  U_2(t,x; c_1,c_2, x_1,x_2)     =   2\sqrt{2} \partial_x \Big[ \arctan \Big(\frac{ e^{s_1} + e^{s_2}  }{1-\rho^2 e^{s_1+s_2}}\Big)\Big],
\]
with $s_1 :=  \sqrt{c_1}(x-c_1t)+x_1$, $s_2:= \sqrt{c_2}(x-c_2t)+x_2$, and $\rho:=\frac{\sqrt{c_1}-\sqrt{c_2}}{\sqrt{c_1}+\sqrt{c_2}}$. Here $c_1,c_2>0$, $c_1\neq c_2$, are the associated scalings, and $x_1,x_2\in \R$ are the corresponding shift parameters. In particular, $U_2$ satisfies
\[
\lim_{t\to \pm \infty} \| U_2(t) - Q_{c_1}(\cdot -c_1 t -x_1^\pm) -Q_{c_2}(\cdot -c_2 t -x_2^\pm)  \|_{H^1(\R)} =0,
\]
for some given $x_j^\pm \in \R$, depending on $(c_1,c_2)$.

\medskip

The study of perturbations of solitons and multi-solitons of (\ref{mKdV}) and more general equations leads to the introduction of the concepts of \emph{orbital}, and \emph{asymptotic stability}. In particular, since energy and mass are conserved quantities, it is natural to expect that solitons are stable in a suitable energy space. Indeed, $H^1$-stability of mKdV and more general solitons and multi-solitons has been considered e.g. in Benjamin \cite{Benj}, Bona-Souganidis-Strauss \cite{BSS}, Weinstein \cite{We2}, Maddocks-Sachs \cite{MS}, Martel-Merle-Tsai \cite{MMT} and Martel-Merle \cite{MMcol2,MMan}. $L^2$-stability of KdV solitons and multi-solitons has been proved in Merle-Vega \cite{MV} and Alejo-Mu\~noz-Vega \cite{AMV}. On the other hand, asymptotic stability properties have been studied by Pego-Weinstein \cite{PW} and Martel-Merle \cite{MMarma,MMnon,MMan}. 

\medskip

One of the main ingredients of the stability argument employed in some of the previous works is the introduction of a suitable \emph{Lyapunov functional}, \emph{invariant or almost invariant in time} and such that the soliton is a corresponding \emph{extremal point}. 
For the mKdV case, this functional is given by
\be\label{H0}
H[u](t) = E[u](t) + c \, M[u](t),
\ee
where $c>0$ is the scaling of the solitary wave, and $E[u]$, $M[u]$ are given in (\ref{M1})-(\ref{E1}). A simple computation shows that for any $z(t)\in H^1(\R)$ small,
\be\label{Expa1}
H[Q_c+z](t) = H[Q_c] + \int_\R z(Q_c''-cQ_c +Q_c^3) + \mathcal Q (t)  + O(\|z(t)\|_{H^1(\R)}^3).
\ee
The first term above is independent of time, while the second one is zero from (\ref{ecQc}). It turns out that the third term $\mathcal Q(t)$ is positive definite modulo two directions, related to the invariance of the equation under shift and scaling transformations (see the second paragraph above). Modulation parameters are then introduced in order to remove those instability modes. Once these directions are controlled, the stability property follows from (\ref{Expa1}).

\medskip

In addition to the special solutions mentioned above, there exists another nonlinear mode, of oscillatory character, known in the physical and mathematical literature as the \emph{breather} solution, and which is a periodic in time, spatially localized real-valued function. Indeed, the following definition is standard (see \cite{W1,La} and references therein):

\begin{defn}\label{mKdVB} 
Let $\al, \bt \in \R\backslash\{0\}$. The breather solution of mKdV \eqref{mKdV} is explicitly given by 
\be\label{breather}
\begin{split}
 B_{\al, \bt}(t,x)   := & 2\sqrt{2} \partial_x \Big[ \arctan \Big( \frac{\bt}{\al}\frac{\sin(\al (x+\delta t))}{\cosh(\bt (x+\ga t))}\Big) \Big]  \\
  = & 2 \sqrt{2} \bt \sech (\bt (x+\ga t)) \Big[Ê\frac{\cos (\al (x+\delta t)) - (\bt /\al) \sin (\al (x+\delta t)) \tanh (\bt (x+\ga t))  }{1 +(\bt/\al)^2 \sin^2 (\al (x+\delta t))  \sech^2 (\bt (x+\ga t)) } \Big] ,
\end{split}
\ee
with 
\be\label{deltagamma}
\delta := \al^2 -3\bt^2, \quad  \ga := 3\al^2 -\bt^2.
\ee
\end{defn}
Note that breathers are periodic in time, but not in space, and this will be essential in our proof. A simple but very important remark is that $\delta \neq \ga$, for all values of $\al$ and $\beta$ different from zero. This means that variables 
$x+\delta t$ and $x+\ga t$ are always independent. Indeed, if $\delta = \ga$, one has from (\ref{deltagamma})
$2(\al^2 +\bt^2) =0,$  which means $\al=\bt=0$, a contradiction.

\medskip

Additionally, note that for each fixed time, the mKdV breather is a function in the Schwartz class, exponentially decreasing in space, with zero mean:
\[
\int_\R B_{\al,\bt} =0.
\]
Moreover, from the scaling invariance, one has $ c^{1/2}B_{\al,\bt}(c^{3/2} t, c^{1/2}x) = B_{c^{1/2}\al, c^{1/2}\bt}(t,x) ,$   for all $c>0$,
and $B_{-\al, \beta} = B_{\al,\beta}$, $B_{\al, -\beta} = -B_{\al, \beta}.$ Therefore, we can assume $\al, \beta>0$, with no loss of generality. Finally, we will denote $\bt$ and $\al$ as the \emph{first} and \emph{second} scaling parameters, and $-\ga$ will be for us the \emph{velocity} of the breather solution.

\medskip

For the sake of completeness, we briefly comment the two limits $\bt/\al\ll1 $ and $\al=0$ in (\ref{breather}). The first one allows to simplify the expression for the breather to
\[
 B_{\al, \bt}(t,x)   \approx 2\sqrt{2}\bt \cos(\al (x+\delta t))\sech (\bt (x+\ga t)) + O\Big(\frac{\bt}{\al}\Big),
\]
and from a qualitative point of view, it shows explicitly its wave packet nature, as an oscillation modulated by an exponentially decaying function (see e.g. Fig. \ref{BreatherAlfas}). The second case is obtained by formally taking the limit $\al \to 0$ in \eqref{breather},
\be\label{polodoble} 
B_{0, \bt}(t,x)  :=  2\sqrt{2} \partial_x \Big[ \arctan \Big( \frac{\bt (x-3\bt^2 t)}{\cosh(\bt (x-\bt^2 t))}\Big) \Big]. 
\ee
This is the well known {\it double pole} solution of mKdV (see e.g. \cite{OW}), which represents  a  soliton-antisoliton pair traveling in the same direction and splitting up at \emph{logarithmic} rate.
\begin{figure}[!htb]
\centering
\includegraphics[width=14.5cm,height=8cm]{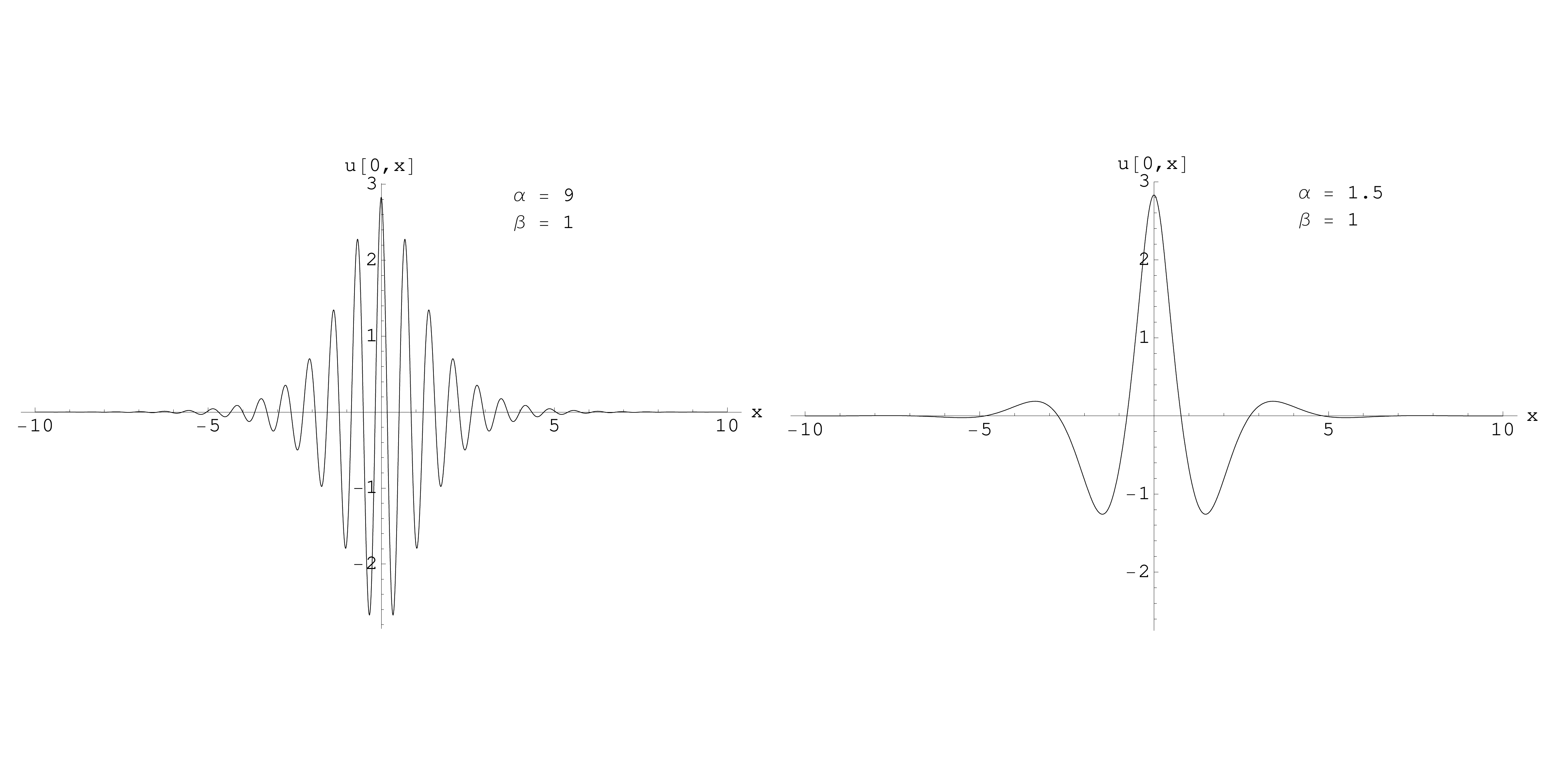}
\small{\caption{Left: mKdV breather \eqref{breather} with $\alpha=9,\beta=1$ at $t=0$.~
Right: mKdV breather \eqref{breather} with $\alpha=1.5,\beta=1$ at $t=0$.\label{BreatherAlfas}}}
\end{figure}

Note that from the invariance under space and time translations, given any $t_0, x_0\in \R$, the function $B_{\al, \bt}(t-t_0,x-x_0)$ is also a breather solution. This fact allows to define a four-parameter family of solutions
\be\label{BB}
B_{\al, \bt} (t , x; x_1, x_2 )  :=  B_{\al, \bt} (t -t_0 , x -x_0)   = 2\sqrt{2} \partial_x \Big[ \arctan \Big( \frac{\bt}{\al}\frac{\sin(\al y_1 )}{\cosh(\bt y_2)}\Big) \Big] ,
\ee
with $y_1 := x+\delta t + x_1$, $y_2:=x+\ga t + x_2$,
\be\label{t0x0}
t_0 := \frac{x_1-x_2}{2(\al^2+\bt^2)}, \quad \hbox{ and } \quad x_0 :=  \frac{\delta x_2-\ga x_1}{2(\al^2 + \bt^2)}.
\ee
Note that from this formula one has, for any $k\in \Z$,
\be\label{dege}
 B_{\al,\bt} (t,x; x_1 + \frac{k\pi}{\al}, x_2)  =  (-1)^k B_{\al,\bt} (t,x; x_1, x_2),
\ee
which are also solutions of (\ref{mKdV}). These identities reveal the periodic character of the first translation parameter. 

\medskip

In the same way, from (\ref{polodoble}) one can define a three-parameter family of double pole solutions $B_{0,\bt}(t; x_1, x_2)$, with $x_1, x_2\in \R$.
\phantom{a}\\
\begin{figure}[!htb]
\centering
\includegraphics[width=14.0cm,height=8.5cm]{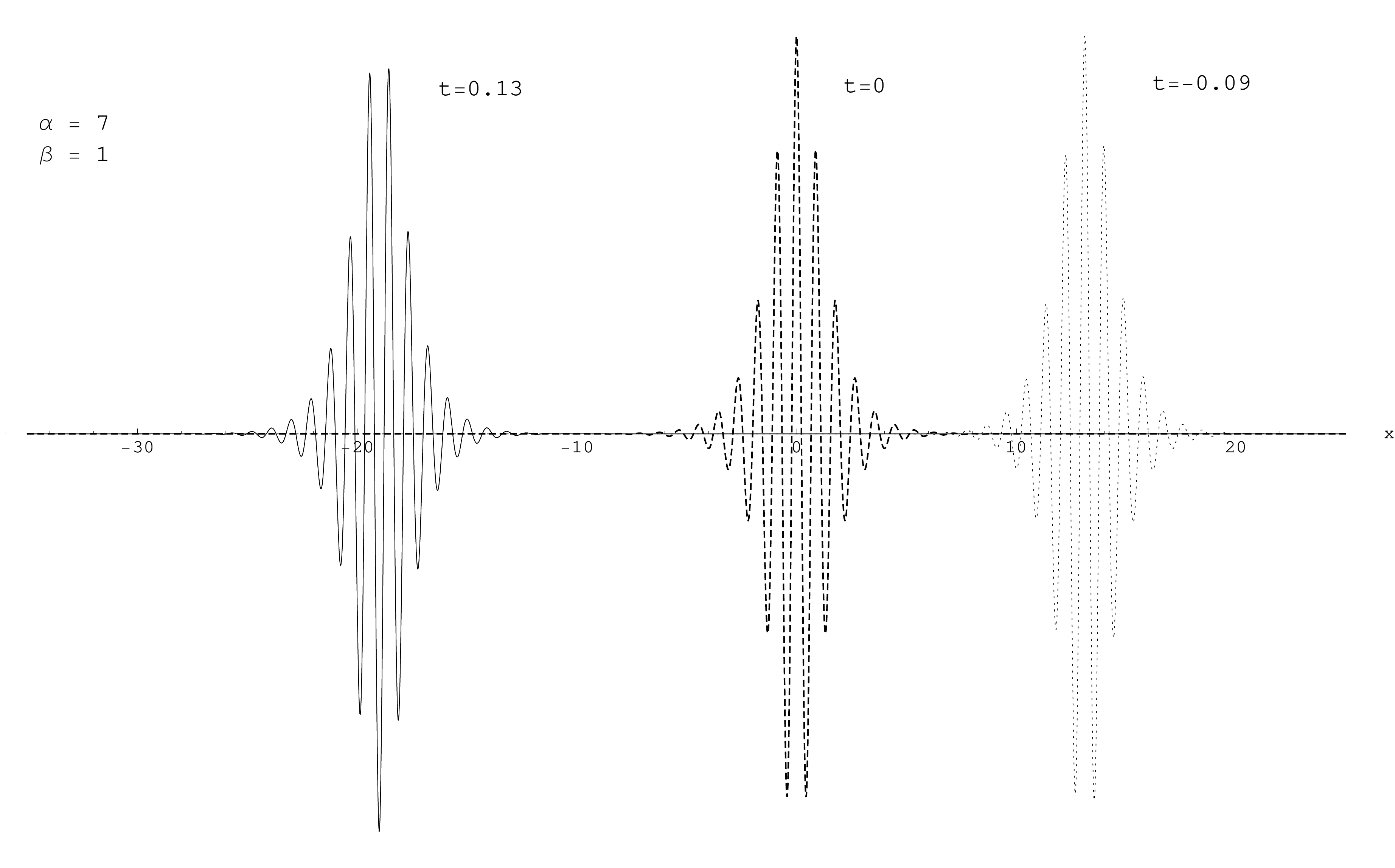}
\small{\caption{ Evolution of the mKdV breather  \eqref{breather} with $\alpha=7,\beta=1$ at instants $t=-0.09$, $t=0$, and $t=0.13$. Note that with the selected values of $\al,\bt$, the \emph{velocity} is given by $\ga=3\al^2 -\bt^2=146>0$ and then the breather moves to the left.\label{EvolBreather}}}
\end{figure}
\phantom{a}\\

Let us come back to breather solutions. We claim that they can be formally associated to the well known mKdV 2-solitons. Indeed, they have a four-parameter family of symmetries: two scaling  and two translations invariances (note that the equation that we consider is just one dimensional in space). However, unlike 2-soliton solutions, breathers have to be considered as \emph{fully bounded states}, since they do not decouple into simple solitons as time evolves. Another intriguing fact is that, as far as we know, breathers are only present in some very particular integrable models, such as mKdV, NLS and sine Gordon equations, among others.

\medskip

Let us recall now some relevant physical and mathematical literature. From the physical point of view, breather solutions are relevant to localization-type phenomena in optics, condensed matter physics and biophysics \cite{Au}. In a geometrical setting, breathers also appear in the evolution of closed planar curves playing the role of  smooth localized deformations traveling along the closed curve \cite{Ale}. Moreover, it is interesting to stress that breather solutions have also been considered by Kenig, Ponce and Vega in their proof of the non-uniform continuity of the mKdV flow in the Sobolev spaces $H^s$, $s<\frac 14$ \cite{KPV2}. On the other hand, they should be essential to completely understand the associated \emph{soliton-resolution} conjecture for the mKdV equation, according to the analysis developed by Schuur in \cite{Sch}. An essential problem in that direction is to show whether or not breather solutions may appear from general initial data, and for this reason to study their stability is the fundamental question. Numerical computations (see Gorria-Alejo-Vega \cite{AGV}) show that breathers are \emph{numerically} stable. However, the simple question of a rigorous proof of orbital stability has become a long standing open problem.

\medskip

In this paper, we give a first, positive answer to the question of breathers stability. Our main result is the following

\begin{thm}\label{MT} mKdV breathers are orbitally stable in their natural $H^2$-topology. 
\end{thm}

A more detailed version of this result is given in Theorem \ref{T1}. As we will see from  the proofs, the space $H^2$ is required by a regularity argument and by the very important fact that breathers are \emph{bound states}, which means that there is no mass decoupling as time evolves. However, our argument is general and can be applied to several equations with breather solutions, and moreover, it introduces several new ideas in order to attack the stability problem in the energy space. In addition, our proof corroborates, at the rigorous level, some deep connections between breathers and the 2-solitons of mKdV.

\medskip

Let us explain the main steps of the proof. First, we prove that breathers satisfy a fourth-order, nonlinear ODE (equation (\ref{EcB})). The proof of this identity is involved, and requires the explicit form of the breather, and several new identities related to the soliton structure of the breather. It seems that this equation cannot be obtained from the original arguments by Lax \cite{LAX1}, since the dynamics do not decouple in time. Our second and more important ingredient is the introduction of a new Lyapunov functional (see \eqref{H1}), well-defined in the $H^2$ topology, and for which breathers are surprisingly not only \emph{extremal points}, but also \emph{local minimizers}, up to  symmetries. This functional also allows to control the perturbative terms and the instability directions that appear during of the dynamics, the latter as consequences of the symmetries described by (\ref{breather}). From the proofs, we will see that breathers have essentially \emph{three directions of instability}, two associated to translation invariances, and a third one consequence of the particular first scaling parameter $\beta$. In order to prove that there is just one negative eigenvalue, we make use of a direct generalization of the theory developed by L. Greenberg \cite{Gr}, which deals with fourth order eigenvalue problems. We then modulate in time in order to remove the spatial instabilities. This is an absolutely necessary condition in order to obtain an orbital stability property. However, we do not modulate the scaling instabilities. Instead, we control the dynamics first replacing the corresponding negative mode by a more tractable direction, the breather itself, and using the mass conservation law. This technique was first introduced by Weinstein in \cite{We1}. A very surprising fact is that the so-called \emph{second scaling} parameter, associated to oscillations, is actually a positive direction when enough regularity is on hand, and even if it has an $L^2$-critical character. 

\medskip

Our functional is reminiscent of that appearing in the  foundational paper by Lax \cite{LAX1}, concerning the 2-soliton solution of the KdV equation,
\[
u_{t}  +  (u_{xx} + u^2)_x =0,
\]
and generalized to the KdV $N$-soliton states by Maddock-Sachs \cite{MS}. This idea has been successfully applied to several 2-soliton problems, for which the dynamics decouples into well-separated solitons as time evolves, see e.g. Holmer-Perelman-Zworski \cite{HPZ}, Kapitula \cite{Kap}, and Lopes-Neves \cite{NL}, for the Benjamin-Ono equation. However, there was no evidence that this technique could be generalized to the case of even more complex solutions, such as breathers. Compared with those results, our proofs are more involved, and computations are sometimes a nightmare. We have preferred to split the proof of the main theorem into several simple steps. 

\medskip

We believe that our result can be improved to reach the $H^1$ level of regularity, but with a harder proof. It seems clear that a better understanding of the $H^1$ dynamics requires a detailed study of modulations on the scaling parameters. In particular, the Martel-Merle-Tsai technique \cite{MMT} seems to fail in this case due to the absence of a clearly decoupled mass dynamics. One can also consider a suitable asymptotic stability property, in the spirit of \cite{MMarma}. However, note that the Martel-Merle \cite{MMarma,MMnon,MMan} results are difficult to generalize to the current case of study since breathers may have negative velocity, and therefore they can interact with the linear part of the dynamics. We conjecture that breathers are asymptotically stable in the case of positive velocities.

\begin{rem}
The methods employed in the proof of Theorem \ref{MT} seem do not apply in the limit $\al\to 0$, which is expected to be unstable, according to the numerical computations performed by Gorria-Alejo-Vega \cite{AGV}. 
\end{rem}

\begin{rem}
The natural complement of our study is to consider the sine Gordon equation
\[
u_{tt}  -u_{xx} + \sin u=0, \quad u(t,x)\in \R.
\]
Since this integrable equation has also breather solutions (see e.g. Lamb \cite{La}), we expect similar results, but with more involved proofs at the level of the linearized problem (we deal with matrix operators).  Indeed, following the present proof, we can guess that sine-Gordon breathers are $H^2\times H^1$ stable provided Lemma \ref{B1B2} and Proposition \ref{WW} hold for the associated spectral elements. Additionally, the focusing Gardner equation 
\[
u_{t} +(u_{xx} +u^2+\mu u^3)_x=0, \quad \mu >0,
\]
is the natural generalization of (\ref{mKdV}). In particular, it has a family of breathers indexed by the additional parameter $\mu$ (see \cite{PeGr, Ale1}). We expect to consider some of these problems in a forthcoming publication (see \cite{AM}).
\end{rem}

In a more qualitative aspect, we think that our results are in some sense a surprise, because any nontrivial perturbation of an integrable equation with breathers solutions should destroy the existence property. Several results in that direction can be found e.g. in \cite{BMW,D,SW} and references therein (for the sine Gordon case). Those results and the present paper suggest that stability is deeply related to the integrability of the equation, unlike the standard gKdV $N$-soliton solution \cite{MMT}.

\medskip

Finally, let us explain the organization of this paper. In Section \ref{2} we study generalized Weinstein conditions satisfied by breather solutions. In Section \ref{3} we prove that any breather profile satisfies a fourth order, nonlinear ODE. Section \ref{4} is devoted to the study of a linear operator associated to the breather solution. In Section \ref{5} we introduce new Lyapunov functional which controls the dynamics. Finally, in Section \ref{6} we prove a detailed version of Theorem \ref{MT}.

\medskip

 \noindent
{\bf Acknowdlegments.} We would like to thank Yvan Martel, Frank Merle, Carlos Kenig and Luis Vega for many useful comments on a first version of this paper.

\bigskip

\section{Stability tests}\label{2}

\medskip

The purpose of this section is to obtain generalized Weinstein conditions for any breather $B$. Indeed, for the case of the mKdV soliton (\ref{Sol}),  the mass (\ref{M1}) and the energy (\ref{E1}) are given by the quantities
\be\label{MQc}
M[Q_c] = \frac 12 c^{1/2}\int_\R Q^2 =  2c^{1/2},
\ee
\be\label{EQc}
E[Q_c] = c^{3/2}E[Q] = -\frac 13 c^{3/2} M[Q] = -\frac 23 c^{3/2}<0.
\ee
These two identities show the explicit dependence of the mass and the energy on the soliton scaling parameter. In particular, the Weinstein condition \cite{We2} reads, for $c>0$,
\be\label{WC}
\partial_c M[Q_c] = c^{-1/2}>0.
\ee
This condition ensures the nonlinear stability of the soliton. We consider now the case of mKdV breathers. Surprisingly enough, the mass of a mKdV breather only depends  on the first scaling parameter $\beta$. In other words, it is independent of $\al$. 

\begin{lem}\label{ME} Let $B=B_{\al,\bt}$ be any mKdV breather, for $\al, \bt>0$. Then
\be\label{MassB}
M[B](t) = 2\bt M[Q]=4\bt .
\ee
\end{lem}
\begin{proof}
We start by writing the breather solution in a more tractable way. From the conservation of mass and invariance under spatial and time translations, we can assume $x_1=x_2 =t=0$ in (\ref{BB}). We have then
\[
B^2(0,x) =  8 \bt^2 \sech^2 (\bt x) \Big[Ê\frac{\cos (\al x) - (\bt /\al) \sin (\al x) \tanh (\bt x)  }{1 +(\bt/\al)^2 \sin^2 (\al x)  \sech^2 (\bt x) } \Big]^2.
\]  
Expanding the square in the numerator, we get after some simplifications
\begin{align*}
 & B^2(0,x)  =   8 \al^2\bt^2 \times \\
  &  \times \Big[Ê\frac{ \al^2 \cosh^2 (\bt x) \cos^2(\al x) +\bt^2 \sin^2(\al x) \sinh^2(\bt x) -2\al\bt \sin(\al x)\cos (\al x) \sinh (\bt x)\cosh (\bt x)  }{(\al^2 \cosh^2 (\bt x) + \bt^2 \sin^2(\al x))^2 } \Big].
\end{align*}
Now the purpose is to use double angle formulas to avoid the squares. More precisely, it is well known that
\be\label{double1}
\cos^2(\al x) =\frac 12 (1+\cos (2\al x)), \quad \sin^2(\al x) =\frac 12 (1-\cos (2\al x)),
\ee
and
\be\label{double2}
\cosh^2 (\bt x) =\frac 12 (1+\cosh(2\bt x)), \quad \sinh^2(\bt x) =\frac 12 (\cosh (2\bt x)-1).
\ee
We replace these identities in the previous expression above. We obtain
\[
B^2(0,x) =Ê\frac{ 8\al^2 \bt^2 h_{\al,\bt}(x)}{(\al^2 +\bt^2  +\al^2 \cosh (2\bt x) - \bt^2 \cos(2\al x))^2 },
\]
with
\begin{align}
h_{\al,\bt}(x) & :=  \al^2 -\bt^2 +(\al^2 +\bt^2) (\cos (2\al x) +\cosh (2\bt x)) \nonu\\
&  \qquad + (\al^2 -\bt^2)\cos(2\al x)\cosh(2\bt x) -2\al\bt \sin(2\al x)\sinh (2\bt x). \label{fab} 
\end{align}
In what follows, let 
\[
f_{\al,\bt}(x):= \al^2 +\bt^2 +\al\bt \sin(2\al x) -\bt^2 \cos(2\al x) +\al^2 (\sinh(2\bt x) + \cosh(2\bt x)),
\]
and
\[
g_{\al,\bt}(x) :=\al^2 +\bt^2  +\al^2 \cosh (2\bt x) - \bt^2 \cos(2\al x).
\]
It is clear that 
\[
f'_{\al,\bt} (x)= 2\al\bt[ \al \cos (2\al x) + \bt \sin(2\al x) +\al \cosh (2\bt x) +\al \sinh(2\bt x)], 
\]
and
\[
g'_{\al,\bt} (x) = 2\al\bt[\bt \sin(2\al x) + \al \sinh(2\bt x)].
\]
Therefore, after a lengthy but direct computation,
\[
 f'_{\al,\bt} (x) g_{\al,\bt} (x) -f_{\al,\bt} (x)g'_{\al,\bt} (x) = 2\al^2 \bt h_{\al,\bt }(x),
\]
and then
\[
B^2(0,x) = 4\bt \frac{f'_{\al,\bt} (x) g_{\al,\bt} (x) -f_{\al,\bt} (x)g'_{\al,\bt} (x)}{g_{\al,\bt} ^2(x)} =4\bt \Big( \frac{f_{\al,\bt}}{g_{\al,\bt}}\Big)'.
\]
In conclusion, we have proved that
\[
\frac 12 \int_{-\infty}^x B^2(0,s)ds = \frac{2\bt [ \al^2 +\bt^2 +\al\bt \sin(2\al x) -\bt^2 \cos(2\al x) +\al^2 (\sinh(2\bt x) + \cosh(2\bt x))]}{\al^2 +\bt^2  +\al^2 \cosh (2\bt x) - \bt^2 \cos(2\al x)}.
\]
Taking limit as $x\to +\infty$, we get the desired conclusion.
\end{proof}

\begin{rem}
Note that the last integral above does not change if we consider a general breather, of the form (\ref{BB}).  Indeed, our proof does not require the time independence of the solution. Then we get
\begin{align}\label{Masss1}
 \mathcal M_{\al,\bt}(t,x)  &:= \frac 12\int_{-\infty}^x B_{\al,\bt}^2(t,s; x_1,x_2)ds \nonu\\
 &~= \frac{2\bt [ \al^2 +\bt^2 +\al\bt \sin(2\al y_1) -\bt^2 \cos(2\al y_1) +\al^2 (\sinh(2\bt y_2) + \cosh(2\bt y_2))]}{\al^2 +\bt^2  +\al^2 \cosh (2\bt y_2) - \bt^2 \cos(2\al y_1)},  
\end{align}
with $y_1$ and $y_2$ defined in (\ref{BB}). This last expression will be useful in Lemma \ref{Id1}.
\end{rem}

A direct consequence of the results above are the following generalized Weinstein conditions:

\begin{cor}\label{WCcor} Let $B= B_{\al,\bt}$ be any mKdV breather of the form \eqref{BB}. Given $t\in \R$ fixed, let
\be\label{LAB}
\Lambda_\al B := \partial_\al B, \quad \hbox{ and }\quad  \Lambda_\bt B := \partial_\bt B.
\ee
Then both functions $\Lambda_\al B$ and $\Lambda_\bt B$ are in the  Schwartz class for the spatial variable, and satisfy the identities
\be\label{LAB1}
\partial_\al M[ B]  = \int_\R B \Lambda_\al B = 0,
\ee
and
\be\label{LAB2}
\partial_\bt M[ B]   = \int_\R B \Lambda_\bt B = 4>0,
\ee
independently  of time.
\end{cor}

\begin{proof}
By simple inspection, one can see that, given $t$ fixed, $\Lambda_\al B_{\al, \bt} $ and $\Lambda_\bt B_{\al, \bt} $ are well-defined Schwartz functions. The proof of (\ref{LAB1}) and (\ref{LAB2}) is consequence of (\ref{MassB}), and the definition of mass (\ref{M1}).
\end{proof}

\begin{rem}
Comparing (\ref{LAB1}) and (\ref{LAB2}) with the Weinstein condition (\ref{WC}), we may think that the second scaling parameter $\al$ is $L^2$-critical. On the opposite side, the first scaling  $\beta$ can be seen as a \emph{stable} parameter.  
\end{rem}

\begin{lem}\label{Id1} Let $B=B_{\al,\bt}$ be any breather of the form \eqref{BB}, with $\al,\bt>0$. Then we have
\ben
\item  $B =\tilde B_x$, with $\tilde B=\tilde B_{\al,\bt}$ given by the smooth $L^\infty$-function
\be\label{tB}
\tilde B(t,x) := 2\sqrt{2}  \arctan \Big( \frac{\bt}{\al}\frac{\sin(\al y_1)}{\cosh(\bt y_2)}\Big). 
\ee
\item For any fixed $t\in \R$, we have $ \tilde B_t$ well-defined in the Schwartz class, satisfiying
\be\label{2nd}
B_{xx} + \tilde B_t + B^3 = 0.
\ee
\item Finally, let $\mathcal M_{\al,\bt}$ be defined by (\ref{Masss1}). Then
\be\label{First}
B_x^2  + \frac 12 B^4 + 2B \tilde B_t - 2(\mathcal M_{\al,\bt})_t=0.
\ee
\een
\end{lem}

\begin{proof}
The first item above is a direct consequence of the definition of $B_{\al,\bt}$ in (\ref{BB}). On the other hand, (\ref{2nd}) is a consequence of (\ref{tB}) and integration in space (from $-\infty$ to $x$) of (\ref{mKdV}). Finally, to obtain (\ref{First}) we multiply (\ref{2nd}) by $B_x$ and integrate in space. 
\end{proof}

\begin{rem}
The reader may compare (\ref{2nd})-(\ref{First}) with the well known identities for the soliton solution of mKdV:
\[
Q_c'' -cQ_c + Q_c^3 =0, \qquad Q_c'^2 -cQ_c^2 +\frac 12 Q_c^4 =0.
\]
\end{rem}

We compute now the energy of a breather solution.

\begin{lem}\label{ME2} Let $B=B_{\al,\bt}$ be any mKdV breather, for $\al, \bt>0$. Then
\be\label{EnergyB}
E[B] =  2\bt (3\al^2-\bt^2) |E[Q]|  = 2\bt \ga |E[Q]|.
\ee
\end{lem}
Let us remark that  the sign of the energy is dictated by the sign of the velocity $\ga$. 

\begin{proof}
First of all, let us prove the following reduction
\be\label{red}
E[B] (t)= \frac 13\int_\R (\mathcal M_{\al,\bt})_t(t,x)dx .
\ee
Indeed,  we multiply (\ref{2nd}) by $B_{\al,\bt}$ and integrate in space: we get
\[
\int_\R B_{x}^2  = \int_\R B \tilde B_t + \int_\R B^4.
\]
On the other hand, integrating (\ref{First}),
\[
\int_\R B_x^2  + \frac 12 \int_\R B^4 + 2 \int_\R B  \tilde B_t - 2 \int_\R (\mathcal M_{\al,\bt})_t=0.
\]
From these two identities, we get
\[
 \int_\R B^4 = \frac 43 \int_\R (\mathcal M_{\al,\bt})_t -2  \int_\R B \tilde B_t ,
\]
and therefore
\[
\int_\R B_{x}^2  = \frac 43 \int_\R (\mathcal M_{\al,\bt})_t -  \int_\R B \tilde B_t .
\]
Finally, replacing the last two identities in (\ref{E1}), we get (\ref{red}), as desired.

\medskip

Now we prove (\ref{EnergyB}).  From (\ref{Masss1}) and similar to the proof of Lemma \ref{ME}, we have
\[
(\mathcal M_{\al,\bt})_t  = 2\bt\frac{   (f_{\al,\bt})_t g_{\al,\bt}  -f_{\al,\bt} (g_{\al,\bt})_t  }{g_{\al,\bt}^2},
\]
where, with a slight abuse of notation, $f_{\al,\bt}$ and $g_{\al,\bt}$ are given now by
\[
f_{\al,\bt} = \al^2 +\bt^2 +\al\bt \sin(2\al y_1) -\bt^2 \cos(2\al y_1) +\al^2 (\sinh(2\bt y_2) + \cosh(2\bt y_2)),
\]
and
\be\label{gab}
g_{\al,\bt} =\al^2 +\bt^2  +\al^2 \cosh (2\bt y_2) - \bt^2 \cos(2\al y_1).
\ee
It is clear that 
\[
(f_{\al,\bt})_t= 2\al\bt[ \al \delta \cos (2\al y_1) + \bt  \delta \sin(2\al y_1) +\al \ga \cosh (2\bt y_2) +\al \ga \sinh(2\bt y_2)], 
\]
and
\[
(g_{\al,\bt})_t = 2\al\bt[\bt \delta \sin(2\al y_1) + \al \ga \sinh(2\bt y_2)].
\]
Therefore
\[
 (f_{\al,\bt})_t g_{\al,\bt}  -f_{\al,\bt} (g_{\al,\bt})_t  = 2\al^2 \bt \tilde h_{\al,\bt },
\]
and
\begin{align}
\tilde h_{\al,\bt} & :=  \ga \al^2 -\delta \bt^2 +(\al^2 +\bt^2) (\delta \cos (2\al y_1) + \ga \cosh (2\bt y_2)) \nonu\\
&  \qquad + (\delta \al^2 - \ga \bt^2)\cos(2\al y_1)\cosh(2\bt y_2) - \al\bt(\delta +\ga) \sin(2\al y_1)\sinh (2\bt y_2).  \label{fab1} 
\end{align}
In conclusion,
\be\label{BMt}
\frac 13(\mathcal M_{\al,\bt})_t  = \frac{  4\al^2 \bt^2 \tilde h_{\al,\bt}  }{3(\al^2 +\bt^2  +\al^2 \cosh (2\bt y_2) - \bt^2 \cos(2\al y_1))^2}.
\ee
Now we split $\tilde h_{\al,\bt}$ into two pieces, according to the parameter $\ga$. From the definition of $\ga$, $\delta$ and $h_{\alpha,\beta}$, we have
\[
\tilde h_{\al,\bt} = \ga h_{\al,\bt} + 2(\al^2 +\bt^2) \hat h_{\al,\bt}, 
\]
where
\[
\hat h_{\al,\bt} :=  \bt^2 -(\al^2 +\bt^2) \cos(2\al y_1) -\al^2 \cos(2\al y_1)\cosh(2\bt y_2) +\al\bt \sin(2\al y_1)\sinh(2\bt y_2).
\]
Note that from Lemma \ref{ME}, more precisely (\ref{Masss1}),
\[
\int_\R  \frac{  4\al^2 \bt^2 \ga  h_{\al,\bt}  }{3(\al^2 +\bt^2  +\al^2 \cosh (2\bt y_2) - \bt^2 \cos(2\al y_1))^2} = \frac 43 \bt \ga  = 2 \bt \ga |E[Q]|,
\]
then, in order to conclude, from (\ref{BMt}) and (\ref{red}) we reduce to prove that 
\be\label{int0}
\int_\R  \frac{ \hat h_{\al,\bt}  }{(\al^2 +\bt^2  +\al^2 \cosh (2\bt y_2) - \bt^2 \cos(2\al y_1))^2} = 0.
\ee
We prove this identity, noting that $(\frac 1{2\al}\sin(2\al y_1))_x = \cos(2\al y_1)$, and from (\ref{gab}),
\begin{align*}
&   g_{\al,\bt}(\frac 1{2\al}\sin(2\al y_1))_x  - \frac 1{2\al}\sin(2\al y_1) (g_{\al,\bt})_x =\\
&  \quad = \cos(2\al y_1)[ \al^2 +\bt^2  +\al^2 \cosh (2\bt y_2) - \bt^2 \cos(2\al y_1) ] -  \bt \sin(2\al y_1) [\bt \sin(2\al y_1) + \al  \sinh(2\bt y_2)]\\
&  \quad = -\bt^2 ( \cos^2(2\al y_1)  + \sin^2(2\al y_1)) +(\al^2 +\bt^2)\cos(2\al y_1) +\al^2 \cos(2\al y_1) \cosh (2\bt y_2) \\
&  \qquad  \qquad  - \al \bt \sin(2\al y_1)   \sinh(2\bt y_2) \\
&  \quad = -\hat h_{\al,\bt}.
\end{align*}
Therefore,
\[
\hbox{ l.h.s. of (\ref{int0})}  = \lim_{x\to +\infty}\frac{ - \sin(2\al y_1)}{2\al(\al^2 +\bt^2  +\al^2 \cosh (2\bt y_2) - \bt^2 \cos(2\al y_1))} =0.
\]
\end{proof}

\begin{rem}
Note that we could follow the approach by Lax \cite[pp. 479--481]{LAX1} to obtain reduced expressions for the mass and energy of a breather solution. However, the resulting terms are actually harder to manage than our  direct approach.
\end{rem}

\begin{cor}\label{WCcor2} Let $B=B_{\al,\bt}$ be any mKdV breather. Then
\be\label{LAB3}
\partial_\al E[B] = 12 \al\bt |E[Q]|>0, \quad \partial_\bt E[B] = 6(\al^2 -\bt^2) |E[Q]|.
\ee
\end{cor}

\bigskip

\section{Nonlinear stationary equations}\label{3}

\medskip

The objective of this section is to prove that any breather profile satisfies a suitable stationary, elliptic equation.

\begin{lem}
Let $B=B_{\al,\bt}$ be any mKdV breather. Then, for all $t\in \R$, 
\be\label{ide1}
 B_{xt} +  2(\mathcal M_{\al,\bt})_t B   = 2(\bt^2 -\al^2) \tilde B_t  +(\al^2 +\bt^2)^2 B.
\ee
\end{lem}
\begin{proof} 
We make use of the explicit expression of the breather. An equivalent for the quantity $(\mathcal M_{\al,\bt})_t $ has been already computed in Lemma \ref{ME2}, see  (\ref{BMt}). On the other hand, from (\ref{tB}) and (\ref{BB}), and using double angle formulas in the denominator, we have,
\be\label{Btexp}
\tilde B_t  = 4\sqrt{2} \al \bt  \Big[Ê\frac{\al \delta \cos (\al y_1)\cosh(\bt y_2) - \bt \ga  \sin (\al y_1) \sinh (\bt y_2)  }{  \al^2 +\bt^2 +\al^2 \cosh(2\bt y_2) -\bt^2 \cos (2\al y_1)  } \Big].
\ee
From (\ref{breather}) we have
\begin{align}
B &  =    2\sqrt{2} \al \bt  \Big[Ê\frac{\al \cos (\al y_1)\cosh(\bt y_2) - \bt  \sin (\al y_1) \sinh (\bt y_2)  }{ \al^2 \cosh^2(\bt y_2) +\bt^2 \sin^2 (\al y_1)  } \Big] \label{Bsimpl0} \\
& = 4\sqrt{2} \al \bt  \Big[Ê\frac{\al \cos (\al y_1)\cosh(\bt y_2) - \bt  \sin (\al y_1) \sinh (\bt y_2)  }{ \al^2 +\bt^2 +\al^2 \cosh(2\bt y_2) -\bt^2 \cos (2\al y_1) } \Big].\label{Bsimpl}
\end{align}
Therefore, from (\ref{Bsimpl0}) and \eqref{BMt},
\[
(\mathcal M_{\al,\bt})_tB  = \frac{16 \sqrt{2} \al^3 \bt^3 \tilde{\theta}(t,x)}{(\al^2 +\bt^2  +\al^2 \cosh (2\bt y_2) - \bt^2 \cos(2\al y_1))^3}
\]
where, with the definition of $\tilde{h}$ in (\ref{fab1}),
\[
\tilde{\theta}(t,x)  :=   \tilde h_{\al,\bt} (\al \cos (\al y_1)\cosh(\bt y_2) - \bt  \sin (\al y_1) \sinh (\bt y_2)).
\]
Let us compute $B_{xt}$. First we have from (\ref{Bsimpl0}),
\[
 B_{x}  = \frac{4 \sqrt{2} \al \bt h_1(t,x)}{(\al^2+\bt^2+\al^2 \cosh(2 \bt y_2)-\bt^2 \cos(2 \al y_1))^2},
\]
where 
\begin{align*}
h_1 (t,x)& :=  -(\al^2+\bt^2) \cosh (\bt y_2)\sin(\al y_1) [\al^2+\bt^2 +\al^2 \cosh(2 \bt y_2)-\bt^2 \cos (2 \al y_1)]  \\
&  -2 \al \bt [\al \cos(\al y_1) \cosh(\bt y_2)-\bt \sin(\al y_1) \sinh(\bt y_2)] [\bt \sin(2 \al y_1)+\al  \sinh(2 \bt y_2)].
\end{align*}
Then,
\[
B_{xt} = \frac{4 \sqrt{2} \al \bt  h_2}{(\al^2+\bt^2+\al^2 \cosh( 2 \bt y_2)-\bt^2 \cos (2 \al y_1))^3},
\]
where, using $g_{\al,\bt}$ previously defined in \eqref{gab},
\[
h_2(t,x)  := -4\al\bt \big[\bt \delta \sin(2 \al y_1)+ \al \ga \sinh (2 \bt y_2) \big] h_{21} + g_{\al,\bt} h_{22},
\]
and
\begin{align*} 
h_{21} (t,x)& :=   -(\al^2+\bt^2) \cosh(\bt y_2) \sin(\al y_1)g_{\al,\bt} \\
&  -2 \al\bt (\al \cos(\al y_1) \cosh(\bt y_2)-\bt \sin(\al y_1) \sinh(\bt y_2)) (\bt \sin(2 \al y_1)+\al \sinh(2 \bt y_2)),
\end{align*}
\begin{align*} 
h_{22}(t,x) := &  \ g_{\al,\bt}\times \Big[\big(-\al \delta (\al^2+\bt^2) \cos(\al y_1) \cosh(\bt y_2) - \bt \ga (\al^2+\bt^2)\sin(\al y_1) \sinh(\bt y_2) \big)g_{\al,\bt} \\
&  -2\al\bt(\al^2+\bt^2)\sin(\al y_1) \cosh(\bt y_2)\big(\al\ga\sinh(2\bt y_2)+\delta\bt\sin(2\al y_1)\big)\\
&  -2\al\bt\big(-\al^2\delta\sin(\al y_1)\cosh(\bt y_2)+\al\bt\ga\cos(\al y_1)\sinh(\bt y_2)\\
&  -\bt\al\delta\cos(\al y_1)\sinh(\bt y_2)-\bt^2\ga\sin(\al y_1)\cosh(\bt y_2)\big)\times\big(\al\sinh(2\bt y_2)+\bt\sin(2\al y_1)\big)\\
&  -4\al^2\bt^2\big(\al\cos(\al y_1)\cosh(\bt y_2)-\bt\sin(\al y_1)\sinh(\bt y_2)\big)\times\big(\ga\cosh(2\bt y_2)+\delta\cos(2\al y_1)\big)\Big].
\end{align*}
Then,
\begin{equation}\label{lhs31}
B_{xt} + 2 (\mathcal M_{\al,\bt})_{t}B=\frac{4\sqrt{2}\al\bt\big[h_2 + 8\al^2\bt^2\tilde{\theta}\big]}{(\al^2+\bt^2+\al^2 \cosh(2 \bt y_2)-\bt^2 \cos(2 \al y_1))^3},
\end{equation}
and recalling that
\begin{align*}
& g_{\al,\bt}^2=(\al^2+\bt^2+\al^2 \cosh(2 \bt y_2)-\bt^2 \cos(2 \al y_1))^2\\
&  = (\al^2+\bt^2)^2 + (\al^4 \cosh(2 \bt y_2)^2+\bt^4 \cos(2 \al y_1)^2-2\al^2\bt^2\cosh(2 \bt y_2)\cos(2 \al y_1))\\
&  - 2(\al^2+\bt^2)(\al^2 \cosh(2 \bt y_2)-\bt^2 \cos(2 \al y_1)),
\end{align*}
collecting terms in \eqref{lhs31} and taking into account \eqref{Bsimpl} and \eqref{Btexp}, after some calculations we have
\begin{align*}
\hbox{r.h.s.  of  }  \eqref{lhs31} & =  \frac{4\sqrt{2}\al\bt}{(\al^2+\bt^2+\al^2 \cosh(2 \bt y_2)-\bt^2 \cos(2 \al y_1))^3} \times \\
&  \Big\{g_{\al,\bt}^2\times\Big[  2(\bt^2-\al^2)\big(\delta\al\cos(\al y_1)\cosh( \bt y_2)-\ga\bt\sin(\al y_1)\sinh(\bt y_2)\big)\\
&  + (\al^2+\bt^2)^2\big(\al\cos(\al y_1)\cosh( \bt y_2)-\bt\sin(\al y_1)\sinh(\bt y_2)\big)  \Big]\Big\}\\
&  = 2(\bt^2 -\al^2) (\tilde B_{\al,\bt})_t  +(\al^2 +\bt^2)^2 B_{\al,\bt}.
\end{align*}
\end{proof}

In what follows, and for the sake of simplicity, we use the notation $B= B_{\al,\bt}$ and $\mathcal M_t = (\mathcal M_{\al,\bt})_t$, if no confusion is present. 

\begin{prop}\label{GB} Let $B= B_{\al,\bt}$ be any mKdV breather. Then, for any  fixed $t\in \R$, $B$ satisfies the nonlinear stationary equation  
\be\label{EcB}
G[B]:= B_{(4x)} -2(\bt^2 -\al^2) (B_{xx} + B^3)  +(\al^2 +\bt^2)^2 B + 5 BB_x^2 + 5B^2 B_{xx} + \frac 32 B^5 =0.
\ee
\end{prop}

\begin{rem}
This identity can be seen as the \emph{nonlinear, stationary equation} satisfied by the breather profile, and therefore it is independent of time and translation parameters $x_1,x_2\in \R$. One can compare with the soliton profile $Q_c(x- ct -x_0)$, which satisfies the standard elliptic equation (\ref{ecQc}), obtained as the first variation of the $H^1$ Weinstein functional (\ref{H0}).
\end{rem}

\begin{proof}[Proof of Proposition \ref{GB}]
From (\ref{2nd}) and (\ref{First}), one has
\begin{align*}
\hbox{l.h.s.  of }  (\ref{EcB}) & = -( B_t + B^3)_{xx} + 2(\bt^2 -\al^2) \tilde B_t  +(\al^2 +\bt^2)^2 B    + 5 BB_x^2  + 5B^2 B_{xx} + \frac 32 B^5 \\
& =   - B_{tx} - BB_x^2 + 2B^2B_{xx} + 2(\bt^2 -\al^2) \tilde B_t  +(\al^2 +\bt^2)^2 B + \frac 32 B^5\\
& = - B_{tx} + B \Big[ \frac 12 B^4 + 2B \tilde B_t -2 \mathcal M_t \Big]  -2B^2 ( \tilde B_t + B^3) + \frac 32 B^5 \\
&  \qquad + 2(\bt^2 -\al^2) \tilde B_t  +(\al^2 +\bt^2)^2 B\\
& = - [ B_{tx} + 2 \mathcal M_t B ] + 2(\bt^2 -\al^2) \tilde B_t  +(\al^2 +\bt^2)^2  B \ = 0.
\end{align*}
In the last line we have used (\ref{ide1}).
\end{proof}

\begin{cor}\label{Cor32}
Let $B^0_{\al,\bt} = B^0_{\al,\bt}(t,x; 0,0)$ be any mKdV breather as in \eqref{breather}, and $x_1(t), x_2(t)\in \R$ two continuous functions, defined for all $t$ in a given interval. Consider the modified breather
\[
 B_{\al,\bt} (t,x):=B^0_{\al,\bt}(t,x; x_1(t),x_2(t)), \qquad (\hbox{cf. \eqref{BB}}). 
\]
Then $B_{\al,\bt} $ satisfies \eqref{EcB}, for all $t$ in the considered interval.
\end{cor}

\begin{proof}
A direct consequence of the invariance of the equation (\ref{EcB}) under spatial translations. Note that (\ref{EcB}) is satisfied even if $B_{\al,\bt}$ is not an exact solution of (\ref{mKdV}).
\end{proof}

\bigskip

\section{Spectral analysis}\label{4}

\medskip

Let $z=z(x)$ be a function to be specified in the following lines. Let $B=B_{\al,\bt}$ be any breather solution, with shift parameters $x_1,x_2$. Let us introduce the following fourth order linear operator:
\begin{align}\label{L1}
\mathcal L [z](x;t) &  :=  z_{(4x)}(x) -2(\bt^2 -\al^2) z_{xx}(x) +(\al^2 +\bt^2)^2 z(x)  + 5B^2 z_{xx}(x) + 10BB_x z_x(x) \nonu \\
&   \qquad  + \ \big[ 5B_x^2  +10 BB_{xx}  + \frac {15}2B^4 -6(\bt^2 -\al^2) B^2 \big] z(x).
\end{align}
In this section we describe the spectrum of this operator. More precisely, our main purpose is to find a suitable coercivity property, independently of the nature of scaling parameters. The main result of this section is contained in Proposition \ref{PropOrtog}. Part of the analysis carried out in this section has been previously introduced by Lax \cite{LAX1}, and Maddocks and Sachs \cite{MS}, so we follow their arguments adapted to the breather case, sketching several proofs.  

\begin{lem} $\mathcal L$ is a linear, unbounded operator in $L^2(\R)$, with dense domain $H^4(\R)$. Moreover, $\mathcal L$ is self-adjoint. 
\end{lem}

It is a surprising fact that $\mathcal L$ is actually self-adjoint, due to the non constant terms appearing in the definition of $\mathcal L$. From standard spectral theory of unbounded operators with rapidly decaying coefficients, it is enough to prove that $\mathcal L^* =\mathcal L$ in $H^4(\R)^2$.

\begin{proof}
Let $z,w\in H^4(\R)$. Integrating by parts, one has
\begin{align*}
\int_\R w \mathcal L[z]  =& \int_\R  w\big[ z_{(4x)} -2(\bt^2 -\al^2) z_{xx} +(\al^2 +\bt^2)^2 z  + 5B^2 z_{xx}  + 10BB_x z_x \big] \\
&    + \int_\R \big[ 5B_x^2  +10 BB_{xx}  + \frac {15}2B^4 -6(\bt^2 -\al^2) B^2 \big] z w\\
=&  \int_\R  \big[ w_{(4x)} - 2(\bt^2 -\al^2) w_{xx} +(\al^2 +\bt^2)^2 w  + (5B^2w)_{xx}   - (10BB_x w)_x  \big]z \\
&    + \int_\R \big[ 5B_x^2  +10 BB_{xx}  + \frac {15}2B^4 -6(\bt^2 -\al^2) B^2 \big] z w  \;  = \int_\R  \mathcal L[w] z.
\end{align*}
Finally, it is clear that $D(\mathcal L^*)$ can be identified with $D(\mathcal L)=H^4(\R)$.
\end{proof}
A consequence of the result above is the fact that the spectrum of $\mathcal L$ is real-valued. Furthermore, the following result describes the continuous spectrum of $\mathcal L$.

\begin{lem} Let $\al,\bt>0$. The operator $\mathcal L$ is a compact perturbation of the constant coefficients operator
\[
\mathcal L_{0} [z]:= z_{(4x)} -2(\bt^2 -\al^2) z_{xx} +(\al^2 +\bt^2)^2 z. 
\]
In particular, the continuous spectrum of $\mathcal L$ is the closed interval $[(\al^2 +\bt^2)^2,+\infty)$ in the case $\beta\geq \al$, and $[ 4\al^2 \bt^2 ,+\infty)$ in the case $\beta< \al$. No embedded eigenvalues are contained in this region.
\end{lem}

\begin{proof}
This result is a consequence of the Weyl Theorem on continuous spectrum.  Let us note that  the nonexistence of embedded eigenvalues (or resonances) is consequence of the rapidly decreasing character of the potentials involved in the definition of $\mathcal L$.  
\end{proof}

\begin{rem}
Note that the condition $\al=\bt$ is equivalent to the identity $\partial_\bt E[B] =0.$ Solitons do not satisfy this last indentity.
\end{rem}

\medskip

We introduce now two directions associated to spatial translations. Let $B_{\al,\beta}$ as defined in (\ref{BB}). We define 
\be\label{B12}
 B_1(t ; x_1,x_2) := \partial_{x_1} B_{\al,\bt}(t ; x_1, x_2),\quad \hbox{ and } \quad  B_2(t ; x_1,x_2): =\partial_{x_2} B_{\al,\bt}(t ;  x_1, x_2).
\ee
It is clear that, for all $t\in \R$ $\al,\bt>0$ and $x_1,x_2\in \R$, both $B_1$ and $B_2$ are real-valued functions in the  Schwartz class, exponentially decreasing in space. Moreover, it is not difficult to see that they are \emph{linearly independent} as functions of the $x$-variable, for all time $t$ fixed.

\begin{lem}\label{B1B2} For each $t\in \R$, one has
\[
\ker \mathcal L =\spawn \big\{ B_1(t;x_1,x_2), B_2(t;x_1,x_2)\big\}.
\]
\end{lem}

\begin{proof}
From Proposition \ref{GB}, one has that $\partial_{x_1}G[B] =\partial_{x_2}G[B] \equiv 0$. Writing down these identities, we obtain
\be\label{LB12}
\mathcal L [B_1](t;x_1,x_2) = \mathcal L [B_{2}](t;x_1,x_2) =0,
\ee
with $\mathcal L$ the linearized operator defined in (\ref{L1}) and $B_1, B_2$ defined in (\ref{B12}). A direct analysis involving ordinary differential equations shows that the null space of $\mathcal L_{0}$ is spawned by functions of the type
\[
e^{\pm \bt x } \cos(\al x), \quad e^{\pm \bt x } \sin(\al x),Ê\quad \al,\bt>0,
\]
(note that this set is linearly independent). Among these four functions, there are only two $L^2$-integrable ones in the semi-infinite line $[0,+\infty)$. Therefore, the null space of $\mathcal L|_{H^4(\R)}$ is spanned by at most two $L^2$-functions. Finally, comparing with (\ref{LB12}), we have the desired conclusion. 
\end{proof}

We consider now the natural modes associated to the scaling parameters, which are the best candidates to generate negative directions for the related quadratic form defined by $\mathcal L$. Recall the definitions of $\Lambda_\al B_{\al,\bt} $ and $\Lambda_\beta B_{\al,\bt} $ introduced in (\ref{LAB}). For these two directions, one has the following
\begin{lem}\label{Scaling} 
Let $B =B_{\al,\bt}$ be any mKdV breather. Consider the scaling directions $\Lambda_\al B$ and $\Lambda_\bt B$ introduced in \eqref{LAB}. Then
\be\label{Pos}
\int_\R  \Lambda_\al B \, \mathcal L [\Lambda_\al B]  =  32 \al^2\bt  >0,
\ee
and
\be\label{Neg}
\int_\R  \Lambda_\bt B\, \mathcal L [\Lambda_\bt B]  =  -32 \al^2 \bt <0.
\ee
\end{lem}

\begin{proof}
 From (\ref{GB}), we get after derivation with respect to $\al$ and $\beta$,
\[
\mathcal L [\Lambda_\al B ] =  -4\al [ B_{xx} + B^3 +(\al^2+\bt^2)B], \qquad \mathcal L [\Lambda_\bt B]  =  4\bt [ B_{xx} + B^3 - (\al^2+\bt^2)B].
\]
We deal with the first identity above. Note that from (\ref{LAB1}), (\ref{E1}) and (\ref{LAB3}),
\[
\int_\R  \Lambda_\al B\, \mathcal L [\Lambda_\al B] =  -4\al  \int_\R [ B_{xx} + B^3 +(\al^2+\bt^2)B] \Lambda_\al B \nonu  =  4\al \partial_\al E[B] >0.
\]
This last identity proves (\ref{Pos}).  Following a similar analysis, and since $E[Q] = -\frac 13 M[Q] = -\frac 23$ (cf. (\ref{MQc})-(\ref{EQc})), one has from (\ref{LAB2}) and (\ref{LAB3}),
\begin{align*}
\int_\R  \Lambda_\bt B \, \mathcal L [\Lambda_\bt B] & =  4\bt  \int_\R [ B_{xx} + B^3 - (\al^2+\bt^2)B] \Lambda_\bt B  \\
& = -4\bt \partial_\bt E[B] - 16\bt (\al^2 +\bt^2) \\
& =  24\bt (\bt^2-\al^2) E[Q] - 16\bt (\al^2 +\bt^2) \  =  -32 \al^2 \bt <0.
\end{align*}
Therefore, (\ref{Neg}) is proved.
\end{proof}
A direct consequence of the identities  above and Corollary \ref{WCcor} is the following result:  
\begin{cor}\label{B0cor}
With the notation of Lemma \ref{Scaling}, let
\be\label{B0}
B_0 :=  \frac{\al\Lambda_{\bt} B + \bt \Lambda_\al B}{8\al\bt (\al^2 +\bt^2)}.
\ee
Then $B_0$ is Schwartz and satisfies $\mathcal L[B_0] = - B$, 
\be\label{negB0}
\int_\R B_0 B = \frac{1}{2\bt (\al^2 +\bt^2)} > 0, \quad \hbox{ and }\quad 
 \frac 12\int_\R B_0\mathcal L [B_0]  = - \frac{1}{4 \bt(\al^2 +\bt^2)}<0.
\ee
\end{cor}

\begin{rem}
In other words, $B_0$ is also a negative direction. Moreover, it is not orthogonal to the breather itself. Note additionally that the constant involved in (\ref{negB0}) is independent of time.
\end{rem}

It turns out that the most important consequence of (\ref{Pos}) is the fact that $\mathcal L$ possesses, for all time, \emph{only one negative eigenvalue}. Indeed, in order to prove that result, we follow the Greenberg and Maddocks-Sachs strategy \cite{Gr,MS}, applied this time to the linear, \emph{oscillatory} operator $\mathcal L$. More specifically, we will use the following

\begin{lem}[Uniqueness criterium, see also \cite{Gr,MS}]\label{Wr1} Let $B=B_{\al,\bt}$ be any mKdV breather, and $B_1,B_2$ the corresponding kernel of the operator $\mathcal L$. Then $\mathcal L$ has 
\[
\sum_{x\in \R} \dim \ker W[B_1, B_2] (t; x)
\] 
negative eigenvalues, counting multiplicity. Here, $W$ is the Wronskian matrix of the functions $B_1$ and $B_2$, 
\be\label{WM}
W[B_1, B_2] (t; x) := \left[ \begin{array}{cc} B_1 & B_2 \\  (B_1)_x & (B_2)_x  \end{array} \right] (t,x).
\ee
\end{lem}

\begin{proof}
This result is essentially contained in \cite[Theorem 2.2]{Gr}, where the finite interval case was considered. As shown in several articles (see e.g. \cite{MS,HPZ}), the extension to the real line is direct and does not require additional efforts. We skip the details.
\end{proof}

In what follows, we compute the Wronskian (\ref{WM}). Contrary to the 2-soliton case, where the decoupling of both solitons at infinity simplifies the proof, here we have carried the computations by hand, because of the coupled character of the breather. The surprising fact is the following greatly simplified expression for (\ref{WM}):

\begin{lem}\label{WMLe}
Let $B=B_{\al,\bt}$ be any mKdV breather, and $B_1, B_2$ the corresponding kernel elements defined in \eqref{B12}.  Then
\be\label{Wsimpl}
\det W[B_1,B_2](t;x) =\frac{ 16\al^3\bt^3 (\al^2 +\bt^2)[ \al\sinh(2 \bt y_2) -\bt \sin(2 \al y_1)]}{(\al^2 +\bt^2 +\al^2 \cosh(2\bt y_2) - \bt^2 \cos(2\al y_1))^2}.
\ee
\end{lem}

\begin{rem}
Since the computation of (\ref{Wsimpl}) involves only  partial derivatives on the $x$-variable, the  result above is still valid for the case of breathers with parameters $x_1,x_2$ depending on time. We skip the details.
\end{rem}

\begin{proof}
We start with a very useful simplification. We claim that
\be\label{Wro}
\det W[B_1,B_2](t;x)  =-2(\al^2 +\bt^2) \int_{-\infty}^x (\tilde B_{12}^2(t,s) - \tilde B_{11}(t,s)\tilde B_{22}(t,s) )ds,
\ee
with $\tilde B=\tilde B(t,x; x_1,x_2)$ defined in (\ref{tB}), and $\tilde B_j =\partial_{x_j}\tilde B$. Let us assume this property. Using (\ref{tB}), we compute each term above.  First of all,
\[
\tilde B_1 = \frac{2\sqrt{2}\al^2\bt \cos(\al y_1)\cosh(\bt y_2) }{\al^2 \cosh^2(\bt y_2) +\bt^2 \sin^2(\al y_1)},  \qquad   \tilde B_2 = \frac{-2\sqrt{2}\al\bt^2 \sin(\al y_1)\sinh(\bt y_2) }{\al^2 \cosh^2(\bt y_2) +\bt^2 \sin^2(\al y_1)}.
\]
Similarly,
\[
\tilde B_{11} =   \frac{-2\sqrt{2}\al^3 \bt \sin(\al y_1) \cosh(\bt y_2)}{\al^2 \cosh^2(\bt y_2) +\bt^2 \sin^2(\al y_1)}   \Big[1 + \frac{2\bt^2  \cos^2(\al y_1)}{\al^2 \cosh^2(\bt y_2) +\bt^2 \sin^2(\al y_1)}\Big],
\]
\[
\tilde B_{12} =   \frac{2\sqrt{2}\al^2 \bt^2\cos(\al y_1) \sinh(\bt y_2)}{\al^2 \cosh^2(\bt y_2) +\bt^2 \sin^2(\al y_1)}  \Big[  1 - \frac{2\al^2 \cosh^2(\bt y_2)}{ \al^2 \cosh^2(\bt y_2) +\bt^2 \sin^2(\al y_1)}\Big],
\]
and
\[
\tilde B_{22} = \frac{ -2\sqrt{2}\al \bt^3  \sin(\al y_1) \cosh(\bt y_2)}{\al^2 \cosh^2(\bt y_2) +\bt^2 \sin^2(\al y_1)}   \Big[1 - \frac{2\al^2 \sinh^2(\bt y_2)}{\al^2 \cosh^2(\bt y_2) +\bt^2 \sin^2(\al y_1)}\Big].
\]
Then,
\begin{align*}
\tilde B_{12}^2 & =  \frac{8\al^4 \bt^4\cos^2(\al y_1) \sinh^2(\bt y_2)}{(\al^2 \cosh^2(\bt y_2) +\bt^2 \sin^2(\al y_1))^2} \times\\
&  \qquad \times \Big[  1 - \frac{4\al^2 \cosh^2(\bt y_2)}{ \al^2 \cosh^2(\bt y_2) +\bt^2 \sin^2(\al y_1)} +\frac{4\al^4 \cosh^4(\bt y_2)}{ (\al^2 \cosh^2(\bt y_2) +\bt^2 \sin^2(\al y_1))^2} \Big],
\end{align*}
and
\begin{align*}
-\tilde B_{11}\tilde B_{22} & = \frac{-8\al^4\bt^4 \sin^2(\al y_1) \cosh^2(\bt y_2)}{(\al^2 \cosh^2(\bt y_2) +\bt^2 \sin^2(\al y_1))^2}  \times \\
&  \qquad \times   \Big[Ê1 + \frac{2\bt^2  \cos^2(\al y_1) -2\al^2 \sinh^2(\bt y_2)}{\al^2 \cosh^2(\bt y_2) +\bt^2 \sin^2(\al y_1)}  -\frac{4\al^2\bt^2  \cos^2(\al y_1) \sinh^2(\bt y_2)}{(\al^2 \cosh^2(\bt y_2) +\bt^2 \sin^2(\al y_1))^2} \Big].
\end{align*}
Adding both terms we obtain, after some simplifications,
\begin{align*}
&  \tilde B_{12}^2 -\tilde B_{11}\tilde B_{22} =\\
&  \qquad =  \frac{8\al^4\bt^4}{(\al^2 \cosh^2(\bt y_2) +\bt^2 \sin^2(\al y_1))^2}   \Big[Ê \cos^2(\al y_1) \sinh^2(\bt y_2) - \sin^2(\al y_1) \cosh^2(\bt y_2) \\
&  \qquad \qquad \qquad + \frac{2\sin^2(\al y_1) \cosh^2(\bt y_2)  [\al^2\sinh^2(\bt y_2) -\bt^2 \cos^2(\al y_1)]}{\al^2 \cosh^2(\bt y_2) +\bt^2 \sin^2(\al y_1)}    \Big]\\
&  \qquad = \frac{8\al^4\bt^4 k_1(t,x)}{(\al^2 \cosh^2(\bt y_2) +\bt^2 \sin^2(\al y_1))^3},
\end{align*}
with 
\begin{align*}
k_1(t,x) & :=Ê  \al^2 \sinh^2(\bt y_2) \cosh^2(\bt y_2) -\al^2   \sin^2(\al y_1) \cosh^2(\bt y_2) \\
&  \qquad -\bt^2 \sin^2(\al y_1) \cosh^2(\bt y_2)  -\bt^2 \sin^2(\al y_1) \cos^2(\al y_1).
\end{align*}
Using double angle formulas, as in (\ref{double1})-(\ref{double2}), we get
\[
 \tilde B_{12}^2 -\tilde B_{11}\tilde B_{22} = \frac{16\al^4\bt^4 k_2}{g_{\al,\bt}^3},
\]
where
\begin{align*}
k_2(t,x) & := \al^2 \sinh^2(2\bt y_2)-\bt^2 \sin^2(2\al y_1) \\
&  \qquad  -(\al^2 +\bt^2)(1+\cosh(2\bt y_2)-\cos(2\al y_1) -\cos(2\al y_1)\cosh(2\bt y_2)).
\end{align*}
and $g_{\al,\bt}$ was defined in (\ref{gab}). The last steps of the proof are the following: since 
\[
k_3(t,x) := 8\al^3 \bt^3 [\bt \sin(2\al y_1) -\al\sinh(2\bt y_2)]
\]
satisfies
\[
(k_3)_x = 16 \al^4 \bt^4 [\cos (2\al y_1) -\cosh(2\bt y_2)],  
\]
and
\[
(k_3)_x  g_{\al,\bt} -2 k_3 (g_{\al,\bt})_x = 16\al^4\bt^4 k_2,
\]
we finally get
\[
 \tilde B_{12}^2 -\tilde B_{11}\tilde B_{22} = \Big( \frac{k_3(t,x)}{g_{\al,\bt}^2} \Big)_x.
\]
Now, with regard to (\ref{Wro}), we integrate in space, to obtain
\[
W[B_1,B_2] = -2(\al^2 +\bt^2)\frac{k_3(t,x)}{g_{\al,\bt}^2} = (\ref{Wsimpl}),
\]
as desired.

\medskip

We prove now (\ref{Wro}). From (\ref{2nd}), taking derivative with respect to $x_1$ and $x_2$, we get
\be\label{eqB1}
(B_1)_{xx} + (\tilde B_1)_t + 3B^2 B_1 =0, \quad (B_2)_{xx} + (\tilde B_2)_t + 3B^2 B_2 =0.
\ee
Multiplying the first equation above by $B_2$ and the second by $-B_1$, and adding both equations, we obtain
\[
(B_1)_{xx}B_2 - (B_2)_{xx} B_1 + (\tilde B_1)_t B_2 -(\tilde B_2)_t B_1 =0,
\]
that is, 
\be\label{interm}
( (B_1)_x B_2 - (B_2)_x B_1)_x =  (\tilde B_2)_t B_1-(\tilde B_1)_t B_2.
\ee
On the other hand, since we are working with smooth functions, one has $B =\tilde B_1 + \tilde B_2$,
\[
B_1 = \tilde B_{11} + \tilde B_{12}, \quad B_2 = \tilde B_{12} + \tilde B_{22},
\]
and
\[
(\tilde B_1)_t = \delta  \tilde B_{11} +  \ga \tilde B_{12}, \quad (\tilde B_2)_t =  \delta\tilde B_{12} + \ga \tilde B_{22}.
\]
Replacing in (\ref{interm}), we get
\[
( (B_1)_x B_2 - (B_2)_x B_1)_x = (\delta- \ga) (\tilde B_{12}^2 -\tilde B_{11}\tilde B_{22}).
\]
Since $\delta = \al^2 -3\bt^2$ and $\ga = 3\al^2 -\bt^2$, substituting above and integrating in space, we obtain the desired conclusion. The proof is complete.
\end{proof}

\begin{prop}\label{WW} The operator $\mathcal L$  defined in \eqref{L1} has a unique negative eigenvalue $-\la_0^2<0$, of multiplicity one. Moreover, $\la_0=\la_0(\al,\bt,x_1,x_2,t)$ depends continuously on its corresponding parameters.
\end{prop}

\begin{proof}
We compute the determinant (\ref{WM}) required by Lemma \ref{Wr1}. From Lemma \ref{WMLe}, after a standard translation argument, we just need to consider the behavior of the function
\be\label{fy2}
f(y_2) =f_{t,\al,\bt,\tilde x_2}(y_2) :=  \al \sinh(2\bt y_2) - \bt \sin (2\al (y_2 + (\delta -\ga)t  + \tilde x_2)),
\ee
for $\tilde x_2 := x_1 -x_2 \in \R$, and $\delta-\ga = -2(\al^2 +\bt^2).$

\medskip

A simple argument shows that for $y_2\in \R$ such that $|\sinh (2\bt y_2)| > \frac \bt\al$, $f$ has no root.
Moreover, there exists $R_0=R_0(\al,\bt)>0$ such that, for all $y_2>R_0$ one has $f(y_2) >0$ and for all $y_2<-R_0$, $f(y_2)<0.$ Therefore, since $f$ is continuous, there is a root $y_0=y_0(t,\al,\bt, \tilde x_2)\in [-R_0, R_0]$ for $f.$ Additionally, if $y_2\neq 0$,
\[
f'(y_2) = 2\al\bt [ \cosh(2\bt y_2) -\cos (2\al (y_2 - 2(\al^2+\bt^2)t + \tilde x_2))]>0,
\]
by simple inspection. Therefore, if $y_0\neq 0$ then it is unique and then
\[
\sum_{x\in \R} \dim \ker W[B_1, B_2] (t; x) = \dim \ker W[B_1,B_2](t; y_0 - \ga t  -x_2) =1,
\]
since $B_1$ or $(B_1)_x$ are not zero at that time. Indeed, it is enough to show that $W[B_1,B_2](t,x)$ is not identically zero, then $\dim \ker W[B_1,B_2]<2$. In order to prove this fact, note that from (\ref{eqB1}) $B_1$ solves, for $t,x_1,x_2\inÊ\R $ fixed, a second order linear ODE with source term $-(\tilde B)_t$. Therefore, by standard well-posedness results, both $B_1$ and $(B_1)_x$ cannot be identically zero at the same point.

Now, let us assume that $y_2=0$ is a zero of $f$. We give a different proof of the same result proved above. From (\ref{fy2}), $t = t_k$ must satisfy the condition
\[
-2(\al^2 +\bt^2) t+ \tilde x_2= \frac{k\pi}{2\al}, \quad k\in \Z,
\] 
(compare with (\ref{dege})). In terms of the variables $y_1$ and $y_2$, one has $y_1 = \frac{k\pi}{2\al}$, $k\in \Z$, and $y_2=0$. Recall that 
\begin{align}
B_1(t,x) & = -2\sqrt{2}\al^2\bt \Big[Ê \frac{\al\sin(\al y_1) \cosh(\bt y_2) + \bt \cos(\al y_1)\sinh(\bt y_2) }{\al^2 \cosh^2(\bt y_2) + \bt^2 \sin^2(\al y_1)} \nonu  \\
&  \qquad + 2\bt^2\sin(\al y_1)\cos(\al y_1) \frac{[ \al\cos(\al y_1) \cosh(\bt y_2) -\bt \sin(\al y_1)\sinh(\bt y_2)] }{(\al^2 \cosh^2(\bt y_2) + \bt^2 \sin^2(\al y_1))^2} \Big]. \label{B1}
\end{align}
Replacing in (\ref{B1}), we get
\[
B_1 (t_k,x) =  \frac{-2\sqrt{2}\al^3\bt \sin(\frac k2\pi)}{\al^2 +\bt^2 \sin^2(\frac k2\pi)} = \begin{cases} \neq 0, & k \hbox{ odd}, \\ 0, & k \hbox{ even}.\end{cases}
\]
For the first case above we conclude as in the previous one. Finally, if $t_k$ satisfies
\[
-2(\al^2 +\bt^2) t+ \tilde x_2= \frac{k\pi}{\al}, \quad k\in \Z,
\]
one has $y_1 =k\pi$, $y_2=0$ and $B_1(t,x) =0$, but from (\ref{B1}), after a direct computation, $(B_1)_x$ is given now by the quantity
\[
(B_1)_x(t_k,x)= -2\sqrt{2}\al^2\bt \Big[ 1  +\frac{\beta^2}{\alpha^2}\cos (k\pi)  +\frac{2\bt^2 \al^2 }{\al^4} \cos^2(k\pi) \Big] \neq 0
\]
for all $k\inÊ\Z$. Therefore $\sum_{x\in \R}\dim \ker W[B_1, B_2] (t_k; x)=1$. In conclusion, for all $t\in \R$,  $\mathcal L$ has just one negative eigenvalue, of multiplicity one. 
\end{proof}

\begin{cor}
There exists a continuous function $f_0 = f_0(\al,\bt)$, well-defined for all $\al,\bt>0$, and such that
\[
-\la_0^2 < -f_0(\al,\bt)<0,
\]
for all $\al,\bt>0$, and all $t,x_1,x_2\in \R$.
\end{cor}
\begin{proof}
This is a consequence of the translation invariance and the fact that $\la_0 $ is a continuous, positive function only depending on $\al,\bt$ and $\tilde x_1:= (\delta-\ga)t + (x_1-x_2)$, periodic in $\tilde x_1$ (and then uniformly positive with respect to $\tilde x_1$). 
\end{proof}

\begin{rem}
Note that the  result above is not clear if we allow $\al,\bt$ depending on  time, as in \cite{HPZ}. Since we do not require any kind of modulation on $\al$ and $\bt$, we can easily conclude in the previous result.
\end{rem}

\medskip

Let $z\in H^2(\R)$, and $B=B_{\al,\bt}$ be any mKdV breather. Let us consider the  quadratic form associated to $\mathcal L$:
\begin{align}
\mathcal Q[z]  := \int_\R z \mathcal L[z] & :=   \int_\R z_{xx}^2 + 2(\bt^2 -\al^2)\int_\R z_{x}^2 +(\al^2 +\bt^2)^2\int_\R z^2 - 5\int_\R B^2 z_x^2 \nonu  \\
&   + 5 \int_\R B_x^2 z^2 +10 \int_\R BB_{xx} z^2  + \frac {15}2 \int_\R B^4 z^2 -6(\bt^2 -\al^2) \int_\R B^2 z^2. \label{QQ}
\end{align}

\begin{rem}
From the definition of $\mathcal Q$ and Lemma \ref{B1B2}, it is clear that $\mathcal Q[B_1] =\mathcal Q[B_2]=0.$ Moreover, 
inequality (\ref{Pos}) means that $\Lambda_\al B$ is actually a positive direction for $\mathcal Q$, a completely unexpected result. Additionally,  from (\ref{QQ}) $\mathcal Q$ is bounded below, namely
\[
\mathcal Q[z] \geq -c_{\al,\bt}\|z\|_{H^2(\R)}^2,
\]
for some positive constant $c_{\al,\bt}$ depending on $\al$ and $\beta$ only. 
\end{rem}

Let $ B_{-1} \in \mathcal S\backslash \{ 0\}$ be an eigenfunction associated to the unique negative eigenvalue of the operator $\mathcal L$, as stated in Proposition \ref{WW}. We assume that $ B_{-1} $ has unit $L^2$-norm, so $B_{-1}$ is now unique. In particular, one has $\mathcal L [ B_{-1}] =-\la_0^2 B_{-1} .$ It is clear from Proposition \ref{WW} and Lemma \ref{B1B2} that the following result holds.

\begin{lem}\label{Coerci}
The eigenvalue zero is isolated. Moreover, there exists a continuous function $\nu_0 =\nu_0(\al,\bt)$, well-defined and positive for all $\al,\bt>0$ and such that, for all $z_0\in H^2(\R)$ satisfying
\be\label{Or1}
\int_\R z_0 B_{-1} =\int_\R z_0 B_1 =\int_\R z_0 B_2 =0,
\ee
then 
\be\label{coee}
\mathcal Q[ z_0] \geq \nu_0\| z_0\|_{H^2(\R)}^2.
\ee
\end{lem}

\begin{proof}
The isolatedness of the zero eigenvalue is a direct consequence of standard elliptic estimates for the eigenvalue problem associated to $\mathcal L$, corresponding uniform convergence on compact subsets of $\R$, and the non degeneracy of the kernel associated to $\mathcal L$. 

\medskip
 
On the other hand, the existence of a \emph{positive} constant $\nu_0=\nu_0(\al,\bt,x_1,x_2,t)$ such that (\ref{coee}) is satisfied is now clear. Moreover, this constant is periodic in $x_1$, continuous in all its variables, and satisfies, via translation invariance, the identity 
\[
\nu_0(\al,\bt,x_1,x_2,t) = \tilde \nu_0(\al,\bt,\tilde x_1)>0, \quad \tilde x_1 := (\delta -\ga)t + (x_1-x_2),  
\]
with $\tilde \nu_0$ continuous in all its variables. Thanks to the periodic character of the variable $\tilde x_1$, we obtain a uniform, positive bound independent of $x_1,x_2$ and $t$, still denoted  $\nu_0$. The proof is complete.
\end{proof}

It turns out that $B_{-1}$ is hard to manipulate; we need a more tractable version of the previous result. 

\begin{prop}\label{PropOrtog} Let $B=B_{\al,\bt}$ be any mKdV breather, and $B_1,B_2$ the corresponding kernel of the associated operator $\mathcal L$. There exists $\mu_0>0$, depending on $\al,\bt$ only, such that, for any  $ z\in H^2(\R)$ satisfying
\be\label{OrthoK}
\int_\R B_1  z = \int_\R B_2 z =0,
\ee
one has
\be\label{Coer}
\mathcal Q[ z] \geq \mu_0\| z\|_{H^2(\R)}^2 -\frac{1}{\mu_0}\Big(\int_\R z B \Big)^2.
\ee
\end{prop}

\begin{proof}
This is a standard result, but we include it for the sake of completeness.  Indeed, it is enough to prove that, under the conditions (\ref{OrthoK}) and the additional orthogonality condition $\int_\R zB =0$, one has
\[
\mathcal Q[ z] \geq \mu_0\| z\|_{H^2(\R)}^2.
\]
In what follows we prove that we can replace $ B_{-1}$ by the breather $B$ in Lemma \ref{Coerci} and the result essentially does not change. Indeed, note that from (\ref{B0}), the function $B_0$ satisfies $\mathcal L[B_0] = - B$, and from (\ref{negB0}),
\be\label{posB0}
\int_\R B_0 B = -\int_\R B_0 \mathcal L[B_0] =-\mathcal Q[B_0] > 0.
\ee
The next step is to decompose $z$ and $B_0$ in $\spawn (B_{-1}, B_1,B_2)$ and the corresponding orthogonal subspace. One has
\[
z =\tilde z +  m B_{-1}, \quad  B_0=  b_0 +  n  B_{-1} + p_1 B_1 + p_2 B_2, \quad m,n, p_1,p_2\in \R,
\]
where 
\[
\int_\R \tilde z  B_{-1} =\int_\R \tilde z B_1=\int_\R \tilde z B_2 =\int_\R   b_0 B_{-1} =\int_\R  b_0 B_1=\int_\R  b_0 B_2 =0. 
\]
Note in addition that 
\[
\int_\R B_{-1}B_1 =\int_\R B_{-1}B_2 =0.
\]
From here and the previous identities we have
\be\label{Qzz}
\mathcal Q[z] =\int_\R (\mathcal L \tilde z - m\la_0^2 B_{-1})(\tilde z +m B_{-1}) = \mathcal Q[\tilde z] -m^2 \la_0^2. 
\ee
Now, since $\mathcal L[B_0] =-B$, one has 
\begin{align}
0 & = \int_\R z B = -\int_\R z\mathcal L[B_0]  =\int_\R \mathcal L[\tilde z +m B_{-1}] B_0\nonu\\
&  = \int_\R (\mathcal L[\tilde z] -m\la_0^2 B_{-1})( b_0 +n B_{-1} + p_1 B_1 + p_2 B_2) \ =  \int_\R \mathcal L[\tilde z]  b_0 -mn\la_0^2.\label{zB}
\end{align}
On the other hand,  from Corollary \ref{B0cor},
\be\label{B0B}
\int_\R B_0 B =  -\int_\R B_0 \mathcal L[B_0] = -\int_\R ( b_0 + n B_{-1}) (\mathcal L[b_0] -n\la_0^2 B_{-1})= -\mathcal Q[b_0] + n^2 \la_0^2. 
\ee
Replacing (\ref{zB}) and (\ref{B0B}) into (\ref{Qzz}), we get
\be\label{deco}
\mathcal Q[z] = \mathcal Q[\tilde z] -\frac{\displaystyle{\Big(\int_\R \mathcal L[\tilde z] b_0\Big)^2}}{\displaystyle{\int_\R B_0B + \mathcal Q[b_0]}}.  
\ee
Note that from (\ref{posB0}) and (\ref{coee}) both quantities in the denominator are positive. Additionally, note that if $\tilde z =\la b_0$, with $\la\neq 0$, then
\[
\Big(\int_\R \mathcal L[\tilde z] b_0 \Big)^2 = \mathcal Q[\tilde z] \mathcal Q[b_0].
\]
In particular, if $\tilde z =\la b_0$,
\be\label{condit}
\frac{\displaystyle{\Big(\int_\R \mathcal L[\tilde z] b_0\Big)^2}}{\displaystyle{\int_\R B_0B + \mathcal Q[b_0]}} \leq a\,  Q[\tilde z], \quad 0<a<1.
\ee
In the general case, using the orthogonal decomposition induced by the scalar product $(\mathcal L \cdot ,Ê\cdot)_{L^2} $ on $\spawn (B_{-1}, B_1,B_2)$, we get the same conclusion as before. Therefore, we have proved (\ref{condit}) for all possible $\tilde z$. 

Finally, replacing in (\ref{deco}) and (\ref{Qzz}), $\mathcal Q[z] \geq (1-a) \mathcal Q[\tilde z] \geq 0$,  and $\mathcal Q[\tilde z] \geq m^2 \la_0^2$. We have, for some $C>0$,
\begin{align*}
\mathcal Q[z] & \geq  (1-a)\mathcal Q[\tilde z]  \geq \frac 12 (1-a)\mathcal Q[\tilde z] + (1-a)m^2 \la_0^2 \\
&\geq  \frac 1C(2 \|\tilde z\|_{H^2(\R)}^2 +  2m^2 \|B_{-1}\|_{H^2(\R)}^2) \geq \frac 1C\|z\|_{H^2(\R)}^2.
\end{align*}
\end{proof}
\bigskip

\section{Lyapunov functional}\label{5}

\medskip
In this section we introduce a new Lyapunov functional for equation (\ref{mKdV}), which will be well-defined at the natural $H^2$ level.  

\medskip

Indeed, let $u_0\in H^2(\R)$ and let $u=u(t) \in H^2(\R)$ be the corresponding local in time solution of the Cauchy problem associated to (\ref{mKdV}), with initial condition $u(0)=u_0$ (cf. \cite{KPV}). Let us define the $H^2$-functional
\be\label{F1}
F[u](t) \!   :=  \! \frac 12 \int_\R u_{xx}^2(t,x) dx -\frac 52 \int_\R u^2(t,x)u_x^2(t,x) dx+ \frac 14 \int_\R u^6(t,x) dx.
\ee
\begin{lem}\label{dF10} Given $u$ local $H^2$-solution of \eqref{mKdV} with initial data $u_0$, the functional $F[u](t)$ is a conserved quantity.  In particular, $u$ is a global-in-time $H^2$-solution.
\end{lem}

The existence of this last conserved quantity is a deep consequence of the \emph{integrability property}. In particular,  it is not present in a general, non-integrable gKdV equation. The verification of Lemma \ref{dF10} is a direct computation.

\medskip

Using the functional $F[u]$ \eqref{F1}, we introduce a new Lyapunov functional specifically related to the breather solution. Let $B = B_{\al,\bt}$ be a mKdV breather, and $t\in \R$, and $M[u]$ and $E[u]$ given in (\ref{M1}), (\ref{E1}). We define
\be\label{H1}
\mathcal{H}[u](t) := F[u](t) + 2(\bt^2-\al^2) E[u](t) + (\al^2 +\bt^2)^2 M[u](t), \quad \al, \, \bt \, \hbox{ scaling parameters}.
\ee
It is clear that $\mathcal H[u]$ represents a  real-valued conserved quantity, well-defined for $H^2$-solutions of (\ref{mKdV}). Moreover, one has the following

\begin{lem}\label{EE0}
Let $z\in H^2(\R)$ be any function with sufficiently small $H^2$-norm, and $B=B_{\al,\bt}$ be any breather solution.  Then, for all $t\in \R$,  one has
\be\label{EE}
\mathcal{H}[B+z] - \mathcal{H}[B]  = \frac 12\mathcal Q[z] + \mathcal N[z],
\ee
with $\mathcal Q$ being the quadratic form defined in \eqref{QQ}, and $\mathcal N[z]$ satisfying $|\mathcal N[z] | \leq K\|z\|_{H^2(\R)}^3.$
\end{lem}

\begin{proof}
We compute:
\begin{align*}
\mathcal{H}[B+z]  = & \frac 12 \int_\R (B+z)_{xx}^2 -\frac 52 \int_\R (B+z)^2(B+z)_x^2 + \frac 14 \int_\R (B+z)^6\\
&  + (\bt^2-\al^2) \int_\R (B+z)_x^2 - \frac 12 (\bt^2 -\al^2) \int_\R (B+z)^4  +\frac 12 (\al^2 +\bt^2)^2 \int_\R (B+z)^2\\
 = & \frac 12 \int_\R B_{xx}^2 -\frac 52 \int_\R B^2B_x^2 + \frac 14 \int_\R B^6\\
&  + (\bt^2-\al^2) \int_\R B_x^2 - \frac 12 (\bt^2 -\al^2) \int_\R B^4  +\frac 12 (\al^2 +\bt^2)^2 \int_\R B^2\\
&   + \int_\R \big[ B_{(4x)} -2(\bt^2 -\al^2) (B_{xx} + B^3)  +(\al^2 +\bt^2)^2 B   + 5B_x^2 B + 5B^2 B_{xx} + \frac 32 B^5\big] z\\
&  +\frac 12 \Big[  \int_\R z_{xx}^2 + 2(\bt^2 -\al^2)\int_\R z_{x}^2 +(\al^2 +\bt^2)^2\int_\R z^2 - 5\int_\R B^2 z_x^2 \nonu  \\
&   + 5 \int_\R B_x^2 z^2 +10 \int_\R BB_{xx} z^2  + \frac {15}2 \int_\R B^4 z^2 -6(\bt^2 -\al^2) \int_\R B^2 z^2 \Big]\\
&  +5\int_\R B^3 z^3 -2(\bt^2-\al^2)\int_\R Bz^3 +\frac 53\int_\R B_{xx} z^3 -5\int_\R B z_x^2 z \\
&  +\frac {15}4 \int_\R B^2 z^4 -\frac 12 (\bt^2-\al^2)\int_\R z^4 -\frac 52\int_\R z^2 z_x^2  +\frac 32 \int_\R B z^5 +\frac 14 \int_\R z^6.
\end{align*}
Therefore, we have the decomposition
\[
\mathcal{H}[B+z]   =  \mathcal{H}[B] + \int_\R G[B] z(t)  + \frac 12\mathcal Q[z] + \mathcal N[z],
\]
where $\mathcal Q$ is defined in (\ref{QQ}), and 
\begin{align*}
G[B] & = & B_{(4x)} -2(\bt^2 -\al^2) (B_{xx} + B^3)  +(\al^2 +\bt^2)^2 B   + 5B_x^2 B + 5B^2 B_{xx} + \frac 32 B^5.
\end{align*}
From Proposition \ref{GB}, one has $G[B] \equiv 0$. Finally, the term $N[z]$ is given by 
\begin{align*}
\mathcal N[z]  := & 5\int_\R B^3 z^3 -2(\bt^2-\al^2)\int_\R Bz^3 +\frac 53\int_\R B_{xx} z^3 -5\int_\R B z_x^2 z \\
 & +\frac {15}4 \int_\R B^2 z^4 -\frac 12 (\bt^2-\al^2)\int_\R z^4 -\frac 52\int_\R z^2 z_x^2  +\frac 32 \int_\R B z^5 +\frac 14 \int_\R z^6.
\end{align*}
Therefore, from direct estimates one has $\mathcal N[z] = O(\|z\|_{H^2(\R)}^3),$ as desired.
\end{proof}
The previous Lemma is the key step to the proof of the main result of this paper, that we develop in the next section.

\bigskip

\section{Proof of the Main Theorem}\label{6}

\medskip

In this section we prove a detailed version of  Theorem \ref{MT}.

\begin{thm}[$H^2$-stability of mKdV breathers]\label{T1} Let $\al, \bt >0$. There exist parameters $\eta_0, A_0$, such that the following holds.  Consider $u_0 \in H^2(\R)$, and assume that there exists $\eta \in (0,\eta_0)$ such that 
\be\label{In}
\|  u_0 - B_{\al,\bt}(0;0,0) \|_{H^2(\R)} \leq \eta.
\ee
Then there exist $x_1(t), x_2(t)\in \R$ such that the solution $u(t)$ of the Cauchy problem for the mKdV equation \eqref{mKdV}, with initial data $u_0$, satisfies
\be\label{Fn1}
\sup_{t\in \R}\big\| u(t) - B_{\al,\bt}(t; x_1(t),x_2(t)) \big\|_{H^2(\R)}\leq A_0 \eta,
\ee
with
\be\label{Fn2}
\sup_{t\in \R}|x_1'(t)| +|x_2'(t)| \leq KA_0 \eta,
\ee
for some constant $K>0$.
\end{thm}

\begin{rem}
The initial condition (\ref{In}) can be replaced by any initial breather profile of the form $\tilde B := B_{\al,\bt}(t_0; x_1^0, x_2^0)$, with $t_0, x_1^0, x_2^0 \in \R$, thanks to the invariance of the equation under translations in time and space. In addition, a similar result is available for the \emph{negative} breather $-B_{\al,\bt}$, which is also a solution of (\ref{mKdV}).

\end{rem}

\begin{proof}[Proof of Theorem \ref{T1}]
\noindent
Let $u_0\in H^2(\R)$ satisfying (\ref{In}), and let $u\in C(\R; H^2(\R))$ be the associated solution of the Cauchy problem (\ref{mKdV}), with initial data $u(0)=u_0$. In what follows, we denote
\[
B =B_{\al,\bt},
\]
if no confusion arises.

\smallskip

We prove the theorem only for positive times, since the negative time case is completely analogous. From the continuity of the mKdV flow for $H^2(\R)$ data, there exists a time $T_0>0$ and continuous parameters $x_1(t), x_2(t)\in \R$, defined for all $t\in [0, T_0]$, and such that the solution $u(t)$ of the Cauchy problem for the mKdV equation (\ref{mKdV}), with initial data $u_0$, satisfies
\be\label{F0}
\sup_{t\in [0, T_0]}\big\| u(t) - B(t; x_1(t),x_2(t)) \big\|_{H^2(\R)}\leq 2 \eta.
\ee
The idea is to prove that $T_0 =+\infty$. In order to do this, let $K^*>2$ be a constant, to be fixed later. Let us suppose, by contradiction, that the \emph{maximal time of stability} $T^*$, namely
\begin{align}\label{Te}
T^* &:=   \sup \Big\{ T>0 \, \big| \, \hbox{ for all } t\in [0, T], \, \hbox{ there exist } \tilde x_1(t), \tilde x_2(t) \in \R  \hbox{ such that } \nonu \\
&  \qquad \qquad \sup_{t\in [0, T]}\big\| u(t) - B(t; \tilde x_1(t), \tilde x_2(t)) \big\|_{H^2(\R)}\leq K^* \eta \Big\},
\end{align}
is finite. It is clear from (\ref{F0}) that $T^*$ is a well-defined quantity. Our idea is to find a suitable contradiction to the assumption $T^*<+\infty.$

\medskip

By taking $\delta_0$ smaller, if necessary, we can apply a well known theory of modulation for the solution $u(t)$.

\begin{lem}[Modulation]\label{Mod} There exists $\eta_0>0$ such that, for all $\eta\in (0, \eta_0)$, the following holds. There exist $C^1$ functions $x_1(t)$, $x_2(t) \in \R$, defined for all $t\in [0, T^*]$, and such that 
\be\label{z}
z(t) := u(t) - {B} (t), \quad B(t,x) := B_{\al,\bt} (t,x; x_1(t),x_2(t)) 
\ee
satisfies, for $t\in [0, T^*]$,
\be\label{Ortho}
\int_\R B_1(t; x_1(t),x_2(t)) z(t) = \int_\R B_2(t; x_1(t),x_2(t)) z(t)=0.
\ee
Moreover, one has
\be\label{apriori}
\|z(t)\|_{H^2(\R)} + |x_1'(t)| + |x_2'(t)| \leq K K^* \eta, \quad \|z(0)\|_{H^2(\R)} \leq K\eta,
\ee
for some constant $K>0$, independent of $K^*$. 
\end{lem}

\begin{proof}
The proof of this result is a classical application of the Implicit Function Theorem.  Let
\[
J_j(u(t), x_1,x_2) := \int_\R (u(t,x) - B (t,x; x_1,x_2)) B_j(t,x; x_1,x_2)dx, \quad j=1,2.
\]
It is clear that $J_j(B(t ;x_1,x_2),x_1,x_2) \equiv 0,$ for all $x_1,x_2\in \R$. On the other hand, one has for $j,k=1,2$,
\[
\partial_{x_k} J_j(u(t), x_1,x_2)\Big|_{(B(t),0,0)} =  - \int_\R B_k (t,x; 0,0) B_j(t,x; 0,0)dx.
\]
Let $J$ be the $2\times 2$ matrix with components $ J_{j,k} := (\partial_{x_k}J_j)_{j,k=1,2}$. From the identity above, one has
\[
\det J = -\Big[ \int_\R B_1^2  \int_\R B_2^2 -  (\int_\R B_1B_2)^2\Big](t;0,0),
\]
which is different from zero from the Cauchy-Schwarz inequality and the fact that $B_1$ and $B_2$ are not parallel for all time. Therefore, in a small $H^2$ neighborhood of $B(t; 0,0)$, $t\in [0,T^*]$ (given by the definition of (\ref{Te})), it is possible to write the decomposition (\ref{z})-(\ref{Ortho}).

\medskip

Now we look at the bounds (\ref{apriori}). The first bounds are consequence of the decomposition itself and the equations satisfied by the derivatives of the scaling parameters, after taking time derivative in (\ref{Ortho}) and using that $\det J\neq 0$. The last bound in (\ref{apriori}) is consequence of (\ref{In}).
\end{proof}

Now, we apply Lemma \ref{EE0} to the function $u(t)$. Since $z(t)$ defined by (\ref{z}) is small, we get from (\ref{EE}) and Corollary \ref{Cor32}:
\be\label{Hut}
\mathcal H[u](t) = \mathcal H[B](t) + \frac 12 \mathcal Q[z](t) + N[z](t).
\ee 
Note that $|N[z](t)|\leq K \|z(t)\|_{H^1(\R)}^3$. On the other hand, by the translation invariance in space,
\[
 \mathcal H[B](t) =\mathcal H[B](0) =\hbox{constant}.
\]
Indeed, from (\ref{t0x0}), we have
\[
B(t,x; x_1(t), x_2(t)) = B(t-t_0(t), x-x_0(t)),
\]
for some specific $t_0,x_0$. Since $\mathcal H$ involves integration in space of polynomial functions on $B, B_x$ and $B_{xx}$, we have 
\[
 \mathcal H[B(t, \cdot ; x_1(t),x_2(t))] = \mathcal H[B(t -t_0(t), \cdot -x_0(t); 0,0)] =   \mathcal H[B(t-t_0(t), \cdot ; 0,0)].
\] 
Finally, $ \mathcal H[B(t-t_0(t), \cdot ; 0,0)] = \mathcal H[B(\cdot , \cdot ; 0,0)] (t-t_0(t))$. Taking time derivative,
\[
 \partial_t \mathcal H[B(t, \cdot ; x_1(t),x_2(t))] =  \mathcal H'[B(\cdot , \cdot ; 0,0)] (t-t_0(t)) \times (1-t_0'(t)) \equiv 0,
\]
hence $\mathcal H[B]$ is constant in time. Now we compare (\ref{Hut}) at times $t=0$ and $t\leq T^*$. From Lemma \ref{dF10} and  (\ref{H1}) we have
\[
\mathcal Q[z](t) \leq  \mathcal Q[z](0) + K\|z(t)\|_{H^2(\R)}^3 +K\|z(0)\|_{H^2(\R)}^3 \leq K\|z(0)\|_{H^2(\R)}^2+K\|z(t)\|_{H^2(\R)}^3. 
\]
Additionally, from (\ref{OrthoK})-(\ref{Coer}) applied this time to the time-dependent function $z(t)$, which satisfies (\ref{Ortho}), we get
\begin{align}
\|z(t)\|_{H^2(\R)}^2&  \leq     K\|z(0)\|_{H^2(\R)}^2 + K\|z(t)\|_{H^2(\R)}^3 + K\abs{\int_\R B(t) z(t)}^2\nonu\\
& \leq  K\eta^2 + K(K^*)^3 \eta^3 +K\abs{\int_\R B(t) z(t)}^2. \label{Map}
\end{align}

\noindent
{\bf Conclusion of the proof.} Using the conservation of mass (\ref{M1}), we have, after expanding $u=B+z$,
\begin{align*}
\abs{\int_\R  B(t)z(t)} &  \leq  K\abs{\int_\R  B(0)z(0)} +K\|z(0)\|_{H^2(\R)}^2+ K\|z(t)\|_{H^2(\R)}^2 \nonu \\
& \leq    K (\eta + (K^*)^2 \eta^2), \quad \hbox{ for each $t\in [0, T^*]$.}
\end{align*}
Replacing this last identity in (\ref{Map}), we get
\[
\|z(t)\|_{H^2(\R)}^2 \leq  K \eta^2 ( 1+ (K^*)^2 \eta^3) \leq \frac 12 (K^*)^2 \eta^2 ,
\]
by taking $K^*$ large enough. This last fact contradicts the definition of $T^*$ and therefore the stability property (\ref{Fn1}) holds true. Finally, (\ref{Fn2}) is a consequence of (\ref{apriori}).
\end{proof}


\bigskip

\begin{thebibliography}{99}


\medskip

\bibitem{AC} M. Ablowitz and P. Clarkson, \emph{Solitons, nonlinear evolution equations and inverse scattering}, London Mathematical Society Lecture Note Series, 149. Cambridge University Press, Cambridge, 1991. 

\bibitem{Ale} M.A. Alejo, \emph{Geometric Breathers of the mKdV Equation}, Acta Appl. Math. DOI 10.1007/s10440-012-9698-y (to appear).


\bibitem{Ale1} M.A. Alejo, \emph{On the ill-posedness of the Gardner equation}, preprint.

\bibitem{AGV} M.A. Alejo, C. Gorria and L. Vega, \emph{Discrete conservation laws and the convergence of long time simulations of the mKdV equation}, preprint. 

\bibitem{AM} M.A. Alejo and C. Mu\~noz, \emph{Nonlinear stability of sine Gordon breathers}, in preparation.

\bibitem{AMV} M.A. Alejo, C. Mu\~noz, and L. Vega, \emph{The Gardner equation and the $L^2$-stability of the $N$-soliton solution of the Korteweg-de Vries equation}, to appear in Transactions of the AMS.

\bibitem{Au} S. Aubry, \emph{Breathers in nonlinear lattices: Existence, linear stability and quantization}, Physica D no. 103 (1997), 201--250. 

\bibitem{Benj} T.B. Benjamin, \emph{The stability of solitary waves}, Proc. Roy. Soc. London A \textbf{328}, (1972) 153--183. 

\bibitem{BMW} B. Birnir, H.P. McKean, and A. Weinstein, \emph{The rigidity of sine-Gordon breathers}, Comm. Pure Appl. Math. \textbf{47}, 1043--1051 (1994).

\bibitem{BSS} J.L. Bona, P. Souganidis, and W. Strauss,  \emph{Stability and instability of solitary waves of Korteweg-de Vries type}, Proc. Roy. Soc. London \textbf{411} (1987), 395--412.

\bibitem{CKSTT} J. Colliander, M. Keel, G. Staffilani, H. Takaoka, and T. Tao,  \emph{Sharp global well-posedness for KdV and modified KdV on $\mathbb R$ and $\mathbb T$}, J. Amer. Math. Soc. \textbf{16} (2003), no. 3, 705--749 (electronic).

\bibitem{D} J. Denzler, \emph{Nonpersistence of breather families for the perturbed Sine-Gordon equation}, Commun. Math. Phys. \textbf{158}, 397--430 (1993).

\bibitem{Ga} C.S. Gardner, M.D. Kruskal, and R. Miura, \emph{Korteweg-de Vries equation and generalizations. II. Existence of conservation laws and constants of motion}, J. Math. Phys. \textbf{9}, no. 8 (1968), 1204--1209.

\bibitem{Gr} L. Greenberg, \emph{An oscillation method for fourth order, self-adjoint, two-point boundary value problems with nonlinear eigenvalues}, SIAM J. Math. Anal. \textbf{22} (1991), no. 4, 1021--1042.  

\bibitem{HIROTA1} R. Hirota, \emph{Exact solution of the \emph{modified} Korteweg-de Vries equation for multiple collisions of solitons}, J. Phys. Soc. Japan, \textbf{33}, no. 5 (1972) 1456--1458.

\bibitem{HPZ} J. Holmer, G. Perelman, and M. Zworski, \emph{Effective dynamics of double solitons for perturbed mKdV}, to appear in Comm. Math. Phys. 

\bibitem{Kap} T. Kapitula, \emph{On the stability of N--solitons in integrable systems}, Nonlinearity, \textbf{20} (2007) 879--907.

\bibitem{KPV} C.E. Kenig, G. Ponce, and L. Vega, \emph{Well-posedness and scattering results for the generalized Korteweg--de Vries equation via the contraction principle}, Comm. Pure Appl. Math. \textbf{46}, (1993) 527--620. 

\bibitem{KPV2} C.E. Kenig, G. Ponce, and L. Vega, \emph{On the ill-posedness of some canonical dispersive equations}, Duke Math. J. \textbf{106} (2001), no. 3, 617--633.

\bibitem{La} G.L. Lamb, \emph{Elements of Soliton Theory}, Pure Appl. Math., Wiley, New York, 1980.

\bibitem{LAX1} P.D. Lax, \emph{Integrals of nonlinear equations of evolution and solitary waves}, Comm. Pure Appl. Math. \textbf{21}, (1968) 467--490.

\bibitem{MS} J.H. Maddocks and R.L. Sachs, \emph{On the stability of KdV multi-solitons}, Comm. Pure Appl. Math. \textbf{46}, 867--901 (1993).

\bibitem{MMarma} Y. Martel and F. Merle, \emph{Asymptotic stability of solitons for subcritical generalized KdV equations}, Arch. Ration. Mech. Anal. \textbf{157} (2001), no. 3, 219--254. 

\bibitem{MMnon} Y. Martel and F. Merle, \emph{Asymptotic stability of solitons of the subcritical gKdV equations revisited}, Nonlinearity \textbf{18} (2005) 55--80.

\bibitem{MMcol2} Y. Martel and F. Merle, \emph{Stability of two soliton collision for nonintegrable gKdV equations}, Comm. Math. Phys. \textbf{286} (2009), 39--79.

\bibitem{MMan} Y. Martel and F. Merle, \emph{Asymptotic stability of solitons of the gKdV equations with general nonlinearity}, Math. Ann. \textbf{341} (2008), no. 2, 391--427.

\bibitem{MMT} Y. Martel, F. Merle, and T.P. Tsai, \emph{Stability and asymptotic stability in the energy space of the sum of $N$ solitons for subcritical gKdV equations}, Comm. Math. Phys. \textbf{231} (2002) 347--373.

\bibitem{MV} F. Merle and L. Vega, \emph{$L^2$ stability of solitons for KdV equation},  Int. Math. Res. Not.  2003,  no. 13, 735--753.

\bibitem{NL} A. Neves and O. Lopes, \emph{Orbital stability of double solitons for the Benjamin-Ono equation}, Commun. Math. Phys. \textbf{262} (2006) 757--791.

\bibitem{OW} K. Ohkuma and M. Wadati, \emph{Multi-pole solutions of the modified Korteweg-de Vries equation}, Journ. Phys. Soc. Japan \textbf{51} no.6 (1982) 2029--2035.

\bibitem{PW} R.L. Pego and M.I. Weinstein, \emph{Asymptotic stability of solitary waves}, Commun. Math. Phys. \textbf{164}, 305--349 (1994).

\bibitem{PeGr} D. Pelinovsky and R. Grimshaw, \textit{ Structural transformation of eigenvalues for a perturbed algebraic soliton potential,} Phys. Lett. A 229 (1997), no.3, 165--172.

\bibitem{Sch} P.C. Schuur, \emph{Asymptotic analysis of soliton problems. An inverse scattering approach}, Lecture Notes in Mathematics, 1232. Springer-Verlag, Berlin, 1986. viii+180 pp.

\bibitem{SW} A. Soffer and M.I. Weinstein, \emph{Resonances, radiation damping and instability in Hamiltonian nonlinear wave equations}, Invent. Math. \textbf{136} (1999), no. 1, 9--74.

\bibitem{W1} M. Wadati, \textit{The modified Korteweg-de Vries Equation}, J. Phys. Soc. Japan, \textbf{34}, no.5, (1973), 1289--1296.

\bibitem{We1} M.I. Weinstein, \emph{Modulational stability of ground states of nonlinear Schr\"odinger equations},  SIAM J. Math. Anal.  \textbf{16}  (1985),  no. 3, 472--491.

\bibitem{We2} M.I. Weinstein, \emph{Lyapunov stability of ground states of nonlinear dispersive evolution equations}, Comm. Pure Appl. Math. \textbf{39}, (1986) 51--68.

\end{thebibliography}
\end{document}